\documentclass[b5paper,10pt,twoside,english,draft]{amsart}
\usepackage[utf8]{inputenc}
\usepackage{mathrsfs}
\usepackage{hyperref}
\usepackage[final]{microtype}

\usepackage[utf8]{inputenc}
\usepackage{amsthm}
\usepackage{amsmath}
\usepackage{eucal}
\usepackage{enumitem}
\usepackage{tikz}
\usepackage[sort,nocompress]{cite}
\usetikzlibrary{matrix,arrows,decorations.pathmorphing}
\usepackage{tikz-cd}
\usepackage{tensor}
\usepackage{enumitem}
\usepackage{mathtools}
\usepackage{amssymb}
\usetikzlibrary{shapes}
\DeclareMathOperator{\add}{\mathsf{add}}

\newcommand\Tcong{\mathrel{\overset{\makebox[0pt]{\mbox{\tiny $T$}}}{\cong}}}
\newcommand\Scong{\mathrel{\overset{\makebox[0pt]{\mbox{\tiny $S$}}}{\cong}}}

\DeclareMathOperator{\bounded}{\mathsf b}

\DeclareMathOperator{\id}{\mathsf{id}}

\DeclareMathOperator{\res}{\mathsf{res}}

\DeclareMathOperator{\op}{\mathsf{op}}

\DeclareMathOperator{\cc}{\mathsf c}

\DeclareMathOperator{\RR}{\mathsf R}

\DeclareMathOperator{\X}{\mathsf X}
\DeclareMathOperator{\Y}{\mathsf Y}
\DeclareMathOperator{\T}{\mathsf T}
\DeclareMathOperator{\SSS}{\mathsf S}
\DeclareMathOperator{\D}{\mathsf D}

\DeclareMathOperator{\A}{\mathsf A}

\DeclareMathOperator{\B}{\mathsf B}

\DeclareMathOperator{\HH}{\mathsf H}
\DeclareMathOperator{\C}{\mathsf C}
\DeclareMathOperator{\K}{\mathsf K}
\DeclareMathOperator{\R}{\mathsf R}
\DeclareMathOperator{\LL}{\mathsf L}

\DeclareMathOperator{\Prod}{\mathsf{Prod}}
\DeclareMathOperator{\Add}{\mathsf{Add}}
\DeclareMathOperator{\maxx}{\mathsf{max}}

\DeclareMathOperator{\Ker}{\mathsf{Ker}}

\DeclareMathOperator{\Cok}{\mathsf{Cok}}
\DeclareMathOperator{\modules}{\mathsf{mod}}
\DeclareMathOperator{\Modules}{\mathsf{Mod}}
\DeclareMathOperator{\Hom}{\mathsf{Hom}}

\DeclareMathOperator{\HOM}{\mathcal{H}\mathsf{om}}
\DeclareMathOperator{\RHom}{\mathsf{\mathbb{R}Hom}}
\DeclareMathOperator{\colim}{\mathsf{colim}}
\DeclareMathOperator{\limit}{\mathsf{lim}}
\DeclareMathOperator{\holim}{\mathsf{holim}}
\DeclareMathOperator{\hocolim}{\mathsf{hocolim}}
\DeclareMathOperator{\Gstable}{\underline{\mathsf{Gproj}}}
\DeclareMathOperator{\GStable}{\underline{\mathsf{GProj}}}
\DeclareMathOperator{\Gcostable}{\overline{\mathsf{Ginj}}}
\DeclareMathOperator{\GcoStable}{\overline{\mathsf{GInj}}}

\DeclareMathOperator{\Ktac}{\mathsf K_{\mathsf{tac}}}

\DeclareMathOperator{\pdim}{\mathsf{pdim}}
\DeclareMathOperator{\Ext}{\mathsf{Ext}}
\DeclareMathOperator{\Tor}{\mathsf{Tor}}

\DeclareMathOperator{\proj}{\mathsf{proj}}
\DeclareMathOperator{\Proj}{\mathsf{Proj}}
\DeclareMathOperator{\inj}{\mathsf{inj}}
\DeclareMathOperator{\Inj}{\mathsf{Inj}}

\DeclareMathOperator{\Db}{\mathsf D^{\mathsf b}}
\DeclareMathOperator{\Dsg}{\mathsf D_{\mathsf{sg}}}

\DeclareMathOperator{\ac}{\mathsf {ac}}

\DeclareMathOperator{\tot}{\mathsf{tot}}
\DeclareMathOperator{\perf}{\mathsf{perf}}

\DeclareMathOperator{\rad}{\mathsf{rad}}
\DeclareMathOperator{\thick}{\mathsf{thick}}
\DeclareMathOperator{\coac}{\mathsf{coac}}

\DeclareMathOperator{\Lambdae}{\Lambda^\mathsf e}

\DeclareMathOperator{\Gproj}{\mathsf{Gproj}}
\DeclareMathOperator{\GProj}{\mathsf{GProj}}
\DeclareMathOperator{\Ginj}{\mathsf{Ginj}}
\DeclareMathOperator{\GInj}{\mathsf{GInj}}

\DeclareMathOperator{\Cone}{\mathsf{Cone}}

\DeclareMathOperator{\pd}{\mathsf{pd}}

\DeclareMathOperator{\coh}{\mathsf{coh}}
\DeclareMathOperator{\Ab}{\mathsf{Ab}}

\numberwithin{equation}{section}
\newtheorem{theorem}{Theorem}[section]
\newtheorem{lemma}[theorem]{Lemma}
\newtheorem{corollary}[theorem]{Corollary}
\newtheorem{proposition}[theorem]{Proposition}
\newtheorem{observation}[theorem]{Observation}
\newtheorem{definition}[theorem]{Definition}
\newtheorem*{theorem*}{Theorem}
\newtheorem*{theorem1}{Theorem I}
\newtheorem*{theorem2}{Theorem II}
\newtheorem*{theorem3}{Theorem III}
\newtheorem*{lemmaappendix1}{Lemma~\ref{lem:compgenff}}
\newtheorem*{lemmaappendix2}{Lemma~\ref{lem:naturalmonosplitartin}}
\newtheorem*{proposition*}{Proposition}
\newtheorem*{definition*}{Definition}
\newtheorem*{lemma*}{Lemma}
\theoremstyle{definition}
\newtheorem{example}[theorem]{Example}

\newtheorem*{construction*}{Construction}
\newtheorem{remark}[theorem]{Remark}

\usepackage[colorinlistoftodos]{todonotes}
\usepackage{hyperref}
\usepackage{color}
\usepackage{cite}

\begin{document}

\title{Change of rings and singularity categories}

\author[Oppermann, Psaroudakis and Stai]{Steffen Oppermann, Chrysostomos Psaroudakis and Torkil Stai}
\address{Institutt for matematiske fag, NTNU, 7491 Trondheim, Norway}
\email{steffen.oppermann@ntnu.no}
\email{torkil.stai@ntnu.no}
\address{Institute of Algebra and Number Theory, University of Stuttgart, Pfaffenwaldring 57, 70569 Stuttgart, Germany}
\email{chrysostomos.psaroudakis@mathematik.uni-stuttgart.de}
%\date{\today}
\thanks{The first and third named authors are supported by Norges forskningsr{\aa}d (grant $250056$). The second named author is supported by Deutsche Forschungsgemeinschaft (DFG, grant KO $1281/14-1$).}
\keywords{Gorenstein projective module, Singularity category, Change of rings, Homotopy category, Acyclic and coacyclic complex, Adjoint functor, $0$-cocompact object, t-Structure, Contravariantly finite subcategory}
\subjclass[2010]{16E30;16E35;16E65;16G;16G50;18E30}

\dedicatory{Dedicated to the memory of Ragnar-Olaf Buchweitz} 

\begin{abstract}
We investigate the behavior of singularity categories and stable categories of Gorenstein projective modules along a morphism of rings. The natural context to approach the problem is via change of rings, that is, the classical adjoint triple between the module categories. In particular, we identify conditions on the change of rings to induce functors between the two singularity categories or the two stable categories of Gorenstein projective modules. Moreover, we study this problem at the level of `big singularity categories' in the sense of Krause~\cite{MR2157133}.
%Our results generalize and extend a related result of Chen~\cite{MR3209319}, formulated in the setting of surjective ring morphisms induced by an idempotent element.
Along the way we establish an explicit construction of a right adjoint functor between certain homotopy categories. This is achieved by introducing the notion of $0$-cocompact objects in triangulated categories and proving a dual version of Bousfield's localization lemma. We provide applications and examples illustrating our main results.
\end{abstract}

\maketitle
\setcounter{tocdepth}{1} \tableofcontents

\section{Introduction}

\subsection*{Singularity categories via morphisms of rings}

The singularity category of a noetherian ring $R$, introduced by Buchweitz in his unpublished manuscript \cite{buchweitz1986maximal} as the Verdier quotient $\Dsg(R)=\Db(\modules R)/\perf R$, is by now a celebrated invariant. This category vanishes precisely when $R$ has finite global dimension and, in a sense, describes how far $R$ is from being regular. Indeed, if $R$ is commutative then $\Dsg(R)$ is a categorical measure for the complexity of the singularities of the spectrum of $R$. It should be remarked that Orlov \cite{MR2101296} later considered $\Db(\coh \mathbb X)/\perf \mathbb X$ for an algebraic variety $\mathbb X$ not only in order to understand the singularities of $\mathbb X$, but also to provide new insight into Kontsevich's homological mirror symmetry conjecture \cite{MR1403918}. Denoting by $\Gstable R$ the stable category of finitely generated Gorenstein projective $R$-modules, there is a natural triangle functor $\Gstable R \hookrightarrow \Dsg(R)$ which is always fully faithful. By a famous theorem due to Buchweitz \cite{buchweitz1986maximal}, obtained independently by Happel in \cite[Theorem~4.6]{MR1112170}, this is even a triangle equivalence provided $R$ is Gorenstein. Notice that this is a more general version of the well known equivalence between the singularity category and the stable module category of a selfinjective algebra, due to Rickard \cite[Theorem~2.1]{MR1027750}.

The notion of singular equivalence of finite dimensional algebras has recently attracted much attention. In particular, Chen \cite{MR2790881,MR2501179,MR2880676,MR3490659} has investigated when certain extensions of rings have equivalent singularity categories. In \cite{MR3274657} this topic was studied from the point of view of recollements. Explicitly, each idempotent element $e$ in $R$ gives rise to a recollement of module categories
\[
\modules R/ \langle e\rangle  \hookrightarrow \modules R \xrightarrow{\pi} \modules e R e,
\]
and the authors of that paper gave necessary and sufficient conditions for the quotient functor $\pi$ to induce a singular equivalence. Moreover, Chen \cite{MR3209319} investigated what is happening in the left hand part of the diagram.

It is thus natural to explore singularity categories and stable categories of Gorenstein projective modules along a more general morphism of rings. Such an $f\colon \Lambda \to \Gamma$ gives rise to change of rings via the classical adjoint triple
\begin{center}
\begin{tikzpicture}
      \matrix(a)[matrix of math nodes,
      row sep=2.5em, column sep=4em,
      text height=1.5ex, text depth=0.25ex]
      { \Modules \Lambda & \Modules \Gamma.  \\};
      \path[->,font=\scriptsize](a-1-2) edge node[above]{$\res$} (a-1-1);
      \path[->,bend left, font=\scriptsize](a-1-1) edge node[above]{$- \otimes_{\Lambda} \Gamma$} (a-1-2);
      \path[->,bend right, font=\scriptsize](a-1-1) edge node[below]{$\Hom_{\Lambda}(\Gamma, -)$} (a-1-2);
\end{tikzpicture}
\end{center}
If moreover the two rings are module finite over a commutative noetherian ring --- i.e.\ they are `noetherian algebras' --- then these functors restrict to an adjoint triple on the level of finitely generated modules. It is our aim to understand when these functors induce functors on the levels of singularity categories and stable categories of Gorenstein projective modules, and to investigate what kind of properties these induced functors have. Our main results in this direction are summarized below.

\begin{theorem1}
   Let $f\colon \Lambda\to\Gamma$ be a morphism of noetherian algebras such that the projective dimension of $\Gamma$ is finite both as left and as right $\Lambda$-module. Then the following hold.
   \begin{enumerate}[label=\emph{(\roman*)}]
      \item We have the solid part of the commutative diagram
      \[ \begin{tikzpicture}
          \matrix(a)[matrix of math nodes,row sep=2.5em, column sep=4em,text height=1.5ex, text depth=0.25ex]
            { \Gstable\Lambda & \Gstable\Gamma  \\
            \Dsg(\Lambda) & \Dsg(\Gamma) \\ };
            \draw[right hook->,shorten <=.3em] (a-1-1) to (a-2-1);
            \draw[right hook->,shorten <=.3em] (a-1-2) to (a-2-2);
            \draw[->,bend left=30] (a-1-1) to node [above] {\scriptsize $- \otimes_{\Lambda} \Gamma$} (a-1-2);
            \draw[->, dashed] (a-1-2) to node[above]{\scriptsize $\res$} (a-1-1);
            \draw[->,bend left=30] (a-2-1) to node [above] {\scriptsize $- \otimes^{\mathbb{L}}_{\Lambda} \Gamma$} (a-2-2);
            \draw[->] (a-2-2) to  node [above] {\scriptsize $\res$} (a-2-1);
            \draw[->, dashed, bend right=30] (a-2-1) to node [below] {\scriptsize $\RHom_{\Lambda}(\Gamma, -)$} (a-2-2);
         \end{tikzpicture} \]
      where the two functors on the bottom level form an adjoint pair.
      \item If moreover $\RHom_{\Lambda}(\Gamma, \Lambda)$ belongs to $\perf \Gamma$ then $\res$ restricts to a functor between stable categories of finitely generated Gorenstein projective modules, and it has a right adjoint $\RHom_{\Lambda}(\Gamma, -)$ on the level of singularity categories, as indicated by the dashed arrows.
      \item If moreover $\Lambda$ is module finite over some regular ring and $\Cone(f)$ belongs to $\perf \Lambdae$, then the functor $-\otimes_\Lambda^{\mathbb L}\Gamma$, hence also $-\otimes_\Lambda\Gamma$, is fully faithful. In particular, in this case the pair $\left(\Dsg(\Lambda),\Ker(\res_{\Dsg}) \right)$is a stable $t$-structure in $\Dsg(\Gamma)$ which restricts to a stable $t$-structure $\left(\Gstable\Lambda, \Ker(\res)\right)$ in $\Gstable\Gamma$ when the condition in \textnormal{(ii)} is met.
   \end{enumerate}
\end{theorem1}

The above theorem is a summary of results of Section~\ref{section:singularitycategories}. We remark that (iii) immediately gives a generalization of the main result in \cite{MR3209319}, by choosing $\Gamma$ to be the quotient of $\Lambda$ by a homological ideal of finite projective dimension as bimodule.

\subsection*{Homotopy categories of big modules}

Often it is necessary to consider all modules, not just finitely generated ones. Also in the realm of singularity categories, `big' categories --- i.e.\ triangulated categories which admit small coproducts --- give a different perspective.

In this case, the category one considers is the homotopy category of acyclic complexes of injective modules, $\K_{\ac}(\Inj \Lambda)$, which should be considered the natural `big singularity category' since it is compactly generated by $\Dsg(\Lambda)$ --- see Krause \cite{MR2157133} and Theorem~\ref{thm:compactcompletions} below. This big singularity category plays a key role in recent developments in the theory of support varieties. In particular, if $R$ is locally a hypersurface ring, then by \cite{MR3234500} the lattice of (compactly generated) localizing subcategories of $\K_{\ac}(\Inj R)$ is isomorphic to the lattice of (specialization closed) subsets of the singular locus of $R$.

As a big version of the stable category of Gorenstein projective modules, it seems natural to consider the homotopy category of totally acyclic complexes of projective (or injective) modules. However, unfortunately, the connection here is not as good as one might have hoped: While the Gorenstein projective modules are always compact in $\Ktac(\Proj \Lambda)$, they only generate this category if the algebra is virtually Gorenstein --- a class of algebras introduced by Beligiannis--Reiten\cite{MR2327478}.
%that is so far difficult to check.
%--- see \cite{MR2737805,MR2524651}.

For these big versions of the categories in Theorem~I we obtain the following result, which is an excerpt of Section~\ref{sec:bigcat} --- see in particular the Summary at the end of the section where we also provide a version in terms of the homotopy categories of coacyclic and totally acyclic complexes of projective modules.

\begin{theorem2}
Let $f\colon \Lambda\to\Gamma$ be a morphism of noetherian algebras such that the projective dimension of $\Gamma$ is finite both as left and as right $\Lambda$-module.
   \begin{enumerate}[label=\emph{(\roman*)}]
      \item There are adjunctions as indicated by the solid arrows in the following diagram.
      \[ \begin{tikzpicture}
       \node (KIA) at (-1,1) {$\K_{\ac}(\Inj \Lambda)$};
       \node (KIB) at (4,1) {$\K_{\ac}(\Inj \Gamma)$};
       \draw [->,bend left=30] (KIA) to (KIB);
       \draw [->,bend right=15] (KIB) to node [above=-.5mm] {$\scriptstyle \lambda \res$} (KIA);
       \draw [->] (KIA) to node [above=-1mm] {$\scriptstyle \Hom_{\Lambda}(\Gamma, -)$} (KIB);
       \draw [->,bend left=15] (KIB) to (KIA);
       \draw [->,bend right=30,dashed] (KIA) to (KIB);
      \end{tikzpicture} \]
      If moreover $\RHom_\Lambda(\Gamma, \Lambda)$ belongs to $\perf\Gamma$, then there is an additional adjoint functor, as indicated by the dashed arrow.
      \item If $\Lambda$ and $\Gamma$ admit dualizing complexes $D_\Lambda$ and $D_\Gamma$, respectively, such that there is an isomorphism $\Hom_\Lambda(\Gamma,D_\Lambda)\cong D_\Gamma$ of complexes of $\Lambda$--$\Gamma$-bimodules, then there are adjunctions as indicated by the solid arrows in the following diagram.
      \[ \begin{tikzpicture}
      \node (KIA) at (-1,1) {$\Ktac(\Inj \Lambda)$};
      \node (KIB) at (4,1) {$\Ktac(\Inj \Gamma)$};
      \draw [->,bend right=15] (KIB) to (KIA);
      \draw [->] (KIA) to node [above=-1mm] {$\scriptstyle \Hom_{\Lambda}(\Gamma, -)$} (KIB);
      \draw [->,bend left=15] (KIB) to (KIA);
      \draw [->,bend right=30,dashed] (KIA) to (KIB);
      \end{tikzpicture} \]
      If moreover $\RHom_\Lambda(\Gamma, \Lambda)$ belongs to $\perf\Gamma$, then there is an additional adjoint functor, as indicated by the dashed arrow.
      \item If moreover the conditions of Theorem~I \textnormal{(iii)} are met, then both occurences of $\Hom_\Lambda(\Gamma,-)$ are fully faithful. In this case the pair $\left(\K_{\ac}(\Inj \Lambda),\Ker(\lambda\res) \right)$ is a stable $t$-structure in $\K_{\ac}(\Inj \Gamma)$ which also restricts to the level of totally acyclic complexes.
   \end{enumerate}
\end{theorem2}

Here $\lambda$ denotes injective resolutions, which become necessary since the restriction of a complex of injectives need not consist of injectives any more. While we do not unveil the precise constructions of all the functors above in this introduction, it is worth mentioning that they are all given completely constructively, at least under the hypotheses of (ii).

\subsection*{Adjoint functors}

In some cases the mere existence of an adjoint functor has important consequences. For instance, the right adjoint of the inclusion functor $\Ktac(\Proj R)\hookrightarrow \K(\Proj R)$, which exists under rather weak restrictions on $R$, was used by Jørgensen \cite{MR2283103} in order to establish the contravariant finiteness of the Gorenstein projective modules in $\Modules R$. In Corollary~\ref{cor:GProjcontfinite} we continue in this direction and add to the class of rings over which such right approximations exist.

On the other hand, for the purpose of doing actual computations, learning from formal arguments that a functor exists is often not satisfactory. In the age of Brown representability, due to Neeman \cite{MR1812507}, such scenarios seem to arise quite often. In certain cases, however, Bousfield's `localization lemma' \cite{MR543337,MR728454} --- as presented by Neeman \cite[Lemma 1.7]{MR1191736} --- tells us precisely what a wealth of left adjoint functors actually look like. To apply his result it suffices to assume that $\T$ is a triangulated category admitting coproducts. The lemma asserts that if $\X$ is a set of compact objects in $\T$, then the inclusion $\X^{\perp} \hookrightarrow \T$ admits a left adjoint and, most remarkably, that the latter is explicitly given by assigning to each object $T$ of $\T$ the homotopy colimit of the cones of iterated right $\Add \X$-approximations of $T$. Such a good grasp on this functor has been proven most useful, e.g.\ in \cite{MR1191736} where massive generalizations of results from algebraic geometry were obtained.

We provide the following dual of Bousfield's lemma. Since non-zero cocompact objects rarely exist in the categories we typically consider, we relax the assumption on the set of objects $\X$ to a notion which we call $0$-cocompactness in order to get an interesting result. 

\begin{theorem3}[Theorem~\ref{theorem.t_from_left_perp}]
Let $\T$ be a triangulated category which admits products. Let $\X$ be a set of $0$-cocompact objects in $\T$  
Then
\[ \left(\prescript{\perp}{}{\X}, (\prescript{\perp}{}{\X})^{\perp}\right) \]
is a stable $t$-structure in $\T$.
\end{theorem3}

For the sake of brevity the reader is referred to Section~\ref{section:rightadjoint} for the precise definition of $0$-cocompactness and a detailed account of the construction itself. 
Here ``$\perp$'' refers to all extensions vanishing.
%In fact, a more precise statement only involving $\Hom$-vanishing holds --- see Theorem~\ref{theorem.t_from_left_perp}.

It should be noted that the above stable $t$-structure provides us with a functorial triangle $T_{\prescript{\perp}{}{\X}} \to T \to T_{\X} \to T_{\prescript{\perp}{}{\X}}[1]$, where the objects $T_{\prescript{\perp}{}{\X}}$ in $\prescript{\perp}{}{\X}$ and $T_{\X}$ in $(\prescript{\perp}{}{\X})^{\perp}$, as well as the maps, are explicitly given by a construction dual to the one of Bousfield. We emphasize that a set of $0$-cocompact objects in a triangulated category cogenerates a stable $t$-structure, in contrast to Bousfield's case. 

We moreover show that the hypotheses of Theorem~III are satisfied in several cases of interest. In particular, we provide explicit descriptions of the right adjoint functors to the following inclusions --- see Corollary~\ref{prop:explicitRforArtin}.
\begin{align*}
& \K_{\ac}(\Inj \Lambda) \hookrightarrow \K(\Inj \Lambda); && \Ktac(\Inj \Lambda) \hookrightarrow \K(\Inj \Lambda);  \\
& \K_{\coac}(\Proj \Lambda) \hookrightarrow \K(\Proj \Lambda); && \Ktac(\Proj \Lambda) \hookrightarrow \K(\Proj \Lambda).
\end{align*}

\subsection*{Outline} Section~\ref{section:preliminaries} recalls and to some extent introduces notions and results that will be employed at later stages. This section consists of seven short subsections. In the first subsection we recall well known concepts and results from the theory of triangle functors and adjoints, as well as homotopy limits and colimits. Our first contributions appear in Subsection~\ref{subsection:projres}, where we provide a description of the right adjoint of the inclusion functor $\K(\Proj R) \hookrightarrow \K(\Modules R)$ and a description of the left adjoint of the embedding $\K(\Inj R) \hookrightarrow \K(\Modules R)$. In Subsection~\ref{subDualizingcomplexes} we discuss dualizing complexes and recall a result of Iyengar--Krause stating that a dualizing complex for a noetherian ring $R$ induces an equivalence between $\K(\Proj R)$ and $\K(\Inj R)$. In Proposition~\ref{prop duality commutes} we show that this equivalence interacts nicely with change of rings. Subsection~\ref{subsection:(co)acyclicity} is devoted to acyclic, coacyclic and totally acyclic complexes: For an Artin algebra $\Lambda$ we show in Proposition~\ref{prop:leftperpforartin} how the subcategories $\K_{\ac}(\Proj \Lambda)$ and $\K_{\coac}(\Proj \Lambda)$ of $\K(\Proj\Lambda)$ can be described as certain left perpendicular classes, and obtain similar descriptions of the subcategories $\K_{\ac}(\Inj \Lambda)$ and $\K_{\coac}(\Inj \Lambda)$ in $\K(\Inj\Lambda)$. In Subsection~\ref{subGorHomAlg} we recall several aspects of Gorenstein homological algebra that are used throughout the paper. In Subsection~\ref{subsection:compactcompletions} we recall in Theorem~\ref{thm:compactcompletions} that if $R$ is a noetherian ring, then the bounded derived category $\Db(\modules R)$ and the singularity category $\Dsg(R)$ can be realized as the subcategories of compact objects in certain compactly generated triangulated categories --- a result due to Krause \cite{MR2157133}. Under additional assumptions on $R$, we show a similar statement for the stable category of Gorenstein projective modules $\Gstable R$. Finally, in Subsection~\ref{lastsubsection} we discuss contravariantly finite subcategories and torsion pairs. Moreover, we recall in Theorem~\ref{thm:jorgensenGProjcontfinite} a result of Jørgensen\cite{MR1859047, MR2283103} providing sufficient conditions on a ring $R$ so that the subcategory of Gorenstein projective modules $\GProj R$ is contravariantly finite in $\Modules R$.

In Section~\ref{section:singularitycategories} we show Theorem~I above by breaking the proof down into several smaller and managable steps, and provide examples that illustrate the result. The claims (i) and (ii) are Theorem~\ref{thm finite main}, while (iii) is Proposition~\ref{prop:coneperfectimpliesfullyfaithful} and Corollary~\ref{cor:twotorsionpair} where we view $f$ as a morphism of bimodules over $\Lambda$. As a byproduct of Theorem~I, we obtain the main result of Chen \cite{MR3209319} as the case $f\colon\Lambda \to \Lambda/I$ for $\Lambda$ a finite dimensional algebra with a homological ideal $I$ of finite projective dimension as bimodule --- see Corollary~\ref{chen'sresult}. As a further application, we show for certain trivial extensions $\Gamma = \Lambda \ltimes X$ that there are fully faithful functors $-\otimes_{\Lambda}^{\mathbb L}\Gamma\colon\Dsg(\Lambda)\to\Dsg(\Gamma)$ and $-\otimes_{\Lambda}\Gamma\colon\Gstable \Lambda \to \Gstable \Gamma$ --- see Corollary~\ref{cor.trivialext}. The results of this section are illustrated with two examples --- see Example~\ref{firstexample} and Example~\ref{secondexample}.

In Section~\ref{sec:bigcat} we extend the functors of Theorem~I to certain homotopy categories of projective and of injective modules. Adjunctions turn out to be plentiful and Theorem~II above reflects only a part of the full picture. In particular, we show that Theorem~I is in fact the restriction of a picture --- see Diagram~\ref{bigpicture} in the Summary of Section~\ref{sec:bigcat} --- that exists on the level of certain homotopy categories, where the singularity category as well as the stable category of Gorenstein projective modules serve as subcategories of compact objects. The claim (i) in Theorem~II is Proposition~\ref{propfivetupleadj} and Proposition~\ref{prop restriction to acyclic}, (ii) is Proposition~\ref{prop:restrictiontotac}, while (iii) is Corollary~\ref{cor:fullyfaithfulonKcoac} and~\ref{corsecondff}. Moreover, we obtain `big' versions of some of the applications of Section~\ref{section:singularitycategories} --- see Corollary~\ref{cor:fullyfaithfulonKcoac} and Corollary~\ref{corsecondff}.

Section~\ref{section approximationsofmodules} deals with right approximations and shows how, for any ring $R$, a contravariantly finite subcategory of $\K(\Proj R)$ induces one in $\Modules R$ --- see Proposition~\ref{prop:contfiniteboundaries}. As a consequence, we slightly expand Jørgensen's result \cite{MR2283103} by showing that the category of Gorenstein projective modules is contravariantly finite in $\Modules R$ as long as $R$ is a noetherian ring with a dualizing complex --- see Corollary~\ref{cor:GProjcontfinite}.

The final Section~\ref{section:rightadjoint} contains a proof of Theorem~III and demonstrates how this result provides explicit descriptions of several functors that appear at earlier stages of the paper. In particular, Theorem~III completes the explicit descriptions of the remaining functors of Diagram~\ref{bigpicture} --- see Corollary~\ref{prop:explicitRforArtin}. 
%The proof first provides us with a torsion pair cogenerated by the set $\X$ of $0$-cocompact objects and then under the closure de-suspension assumption on $\X$, we show that $(\prescript{\perp_0}{}{\X}, (\prescript{\perp_0}{}{\X})^{\perp_0})$ is a t-structure. Moreover, we show that  $(\prescript{\perp_0}{}{\X})^{\perp_0}$ is the smallest subcategory of $\T$ containing $\X$ and closed under extensions and products.
%The proof requires some basic facts from the theory of derivators and Lemma~\ref{dualKellerNicolas}, which is the dual of a result of Keller--Nicolás \cite{MR3031826}. 
We also discuss our notion of a $0$-cocompact object in a triangulated category and, in particular, show in Corollary~\ref{corollary.artinina_0-cocompact} that if $\Lambda$ is an Artin algebra, then each finite complex of finitely generated $\Lambda$-modules is $0$-cocompact in the homotopy category $\K(\Modules \Lambda)$.

\subsection*{Notation} By a \emph{noetherian algebra} we mean a ring which is module finite over some commutative noetherian ring. By a module we mean a right module. When $\A$ is an additive category and $\X$ is a class of objects in $\A$, the left and right \emph{orthogonal classes} of $\X$ are the subcategories
\[
{}^{\perp}\X = \{A\in\A \vert \A(A,\X)=0\} \text{ and } \X^{\perp} = \{A\in\A \vert \A(\X,A)=0\},
\]
respectively. When $\T$ is a triangulated category and $\X$ is a class of objects in $\T$, we write
\[
{}^{\perp}\X = \{T\in\T \vert \T(T,\X[i])=0 \text{ $\forall$ } i\} \text{ and } \X^{\perp} = \{T\in\T \vert \T(\X[i],T)=0\text{ $\forall$ } i\}.
\]
If $\A$ has coproducts, then $\Add \X$ is the subcategory consisting of summands of coproducts of objects in $\X$. Dually, if $\A$ has products, then $\Prod \X $ denotes the subcategory of summands of products of objects in $\X$.

%\begin{ackn}
%The first and third authors are supported by Norges forskningsrad grant $250056$. The second author is supported by Deutsche Forschungsgemeinschaft (DFG, grant KO $1281/14-1$).
%\end{ackn}

\section{Preliminaries} \label{section:preliminaries}
\subsection{Triangle functors and adjoints} Let us collect a few facts about functors between triangulated categories that will be employed in the sequel.

\begin{lemma}[\textnormal{Orlov \cite[Lemma~1.2]{MR2101296}}]
\label{lemadjointquotient}
Let $F\colon \T\to \T'$ be an exact functor between triangulated categories with right adjoint $G$. Let $\X$ and $\X'$ be triangulated subcategories of $\T$ and $\T'$, respectively, satisfying $F(\X)\subseteq \X'$ and $G(\X')\subseteq \X$. Then the induced functors $\overline{F}\colon \T/\X\to \T'/\X'$ and $\overline{G}\colon \T'/\X'\to \T/\X$ form an adjoint pair $(\overline{F}, \overline{G})$. Moreover, if $F$ is fully faithful, then so is $\overline{F}$.
\end{lemma}

We now need to recall a few key notions. Suppose $\T$ is a triangulated category admitting coproducts. A triangulated subcategory of $\T$ is called \emph{thick} if it is closed under direct summands, and \emph{localizing} if it is closed under coproducts. It is not hard to show that a localizing subcategory is automatically thick. An object $C$ in $\T$ is called \emph{compact} if the functor $\T(C,-)$ preserves coproducts, i.e.\ if any morphism $C\to\coprod T_i$ factors through a finite subcoproduct. $\T$ is \emph{compactly generated} if it admits a set $\C$ of compact objects that generate $\T$ in the sense that $\T(\C,X)=0$ implies $X=0$. In this case $\T$ coincides with the smallest localizing subcategory containing $\C$. Moreover, the subcategory $\T^{\cc}$ of all compact objects in $\T$ coincides with the smallest thick subcategory containing $\C$.

The following theorem is a consequence of Brown representability for compactly generated triangulated categories \cite[Theorem 8.6.1]{MR1812507}.

\begin{theorem} [{{Neeman \cite[Theorems~4.1 and 5.1]{MR1308405} and \cite[Theorem~8.6.1]{MR1812507}}}] \label{thm:neemanthreeadjoints}
Suppose $F\colon \T\to\T'$ is an exact functor between triangulated categories with $\T$ compactly generated.
\begin{enumerate}[label=\emph{(\roman*)}]
\item $F$ admits a right adjoint if and only if it preserves coproducts.
\item $F$ admits a left adjoint if and only if it preserves products.
\item If $F$ admits a right adjoint $G$, then $F$ preserves compact objects if and only if $G$ preserves coproducts.
\end{enumerate}
\end{theorem}

The proof of the following observation is a standard dévissage argument and can be found in the Appendix.

\begin{lemma} \label{lem:compgenff}
   Let $\T$ and $\T'$ be triangulated categories with coproducts and suppose $\T$ is compactly generated. Let $F \colon\T\to\T'$ be exact and coproduct preserving. If $F$ restricts to a fully faithful functor $\T^{\cc}\to \T'^{\cc}$, then $F$ is fully faithful.
\end{lemma}

%\subsection{Homotopy limits and colimits}
%\label{subsectionholim}

As the only monomorphisms and epimorphisms in a triangulated category are the ones that are split, limits and colimits rarely exist. In certain cases the following machinery is still useful. Suppose we are given a sequence
\[
T_0 \xrightarrow{t_0} T_1 \xrightarrow{t_1} T_2 \to \cdots
\]
in a triangulated category admitting coproducts, and denote by $t\colon\coprod T_i \to \coprod T_i$ the morphism induced by the $t_i$. The \emph{homotopy colimit} of such a sequence, a notion which originates from algebraic topology and was introduced in algebra through \cite{MR1214458}, is defined by the triangle
\[
\coprod T_i \xrightarrow{1-t} \coprod T_i \to \hocolim T_i \to \left(\coprod T_i \right)\![1].
\]
In particular, taking $\hocolim$ is not in general a functorial procedure, but it does commute with left adjoint functors. For each $T_j$ there is a morphism $\mu_j \colon T_j \to \hocolim T_i$ which is compatible with the $t_i$. It is easy to check that if $t_i$ is invertible for each $i$, then each $\mu_j$ is an isomorphism. Moreover, if each $t_i=0$, then $\hocolim T_i$ vanishes. Using homotopy colimits, it is not difficult to show that a triangulated category with coproducts is automatically idempotent complete. Homotopy colimits also (often) enable us to construct the totalization of a complex in a triangulated category. This should be thought of as an analog of the fundamental notion of the total complex with respect to coproducts of a double complex.

Dually, let $\T$ be a triangulated category with products and consider a sequence
\[
\cdots \xrightarrow{} S_2 \xrightarrow{s_2} S_1 \xrightarrow{s_1} S_0.
\]
The {\em homotopy limit} of this sequence is given by the triangle
\[
\left(\prod S_i\right)\![-1] \to \holim{S_i} \to \prod S_i \xrightarrow{1-s} \prod S_i,
\]
where the $i$-th component of $s$ is the map $S_i\xrightarrow{s_i} S_{i-1}\to \prod S_i$. Note again that the $\holim$ is determined up to a non-unique isomorphism.

\subsection{Projective and injective resolutions}
\label{subsection:projres}

Let $R$ be a ring. The homotopy category $\K(\Proj R)$ is not in general compactly generated, but by \cite{MR2439608} it is always `well generated' which means that it still satisfies Brown representability. Hence the coproduct preserving inclusion $\K(\Proj R) \hookrightarrow \K(\Modules R)$ has a right adjoint, which we will denote by $\rho$.
Similarly, by \cite{MR3212862} the inclusion $\K(\Inj R) \hookrightarrow \K(\Modules R)$ has a left adjoint which we will denote by $\lambda$.

The aim of this subsection is twofold: First we give explicit descriptions of $\rho$ and $\lambda$ for complexes whose terms have bounded homological dimension --- see Proposition~\ref{prop nice proj resol}. Then, in Proposition~\ref{prop syzygy-cosyzygy}, we provide a somewhat surprising syzygy-cosyzygy adjunction which will prove to be useful later.

Both these results hinge on the following slightly technical lemma.

\begin{lemma} \label{lem.short_ex}
Let $0 \to X \to C \to Y\to 0$ be a short exact sequence in $\C(\Modules R)$, and assume $C$ is contractible. (Note that such a sequence typically does not give rise to a triangle in $\K(\Modules R)$.)
\begin{enumerate}[label=\emph{(\roman*)}]
\item If $M\in\C(\Modules R)$ is such that $\Ext_R^1(M^i, X^j) = 0$ for all $i$ and $j$, then
\[ \Hom_{\K(\Modules R)}(M, X) \cong \Hom_{\K(\Modules R)}(M, Y[-1]). \]
\item If $N\in\C(\Modules R)$ is such that $\Ext_R^1(Y^i, N^j) = 0$ for all $i$ and $j$, then
\[ \Hom_{\K(\Modules R)}(X, N) \cong \Hom_{\K(\Modules R)}(Y[-1], N). \]
\end{enumerate}
\end{lemma}

\begin{proof}
We only prove the first statement, the second one is dual.

The vanishing of $\Ext^1$ means that $0\to X \to C \to Y\to 0$ induces a short exact sequence
\[ 0\to \HOM(M, X) \to \HOM(M, C) \to \HOM(M, Y) \to 0. \]
Since the middle term is still contractible, this short exact sequence reveals a quasi-isomorphism $\HOM(M, Y)[-1] \to \HOM(M, X)$. The claim follows, as the morphisms in the homotopy category appear as the homologies of these $\HOM$-complexes.
\end{proof}

Recall that the (abelian) category of complexes $\C(\Modules R)$ has enough projectives. Indeed, the $i$-th term $X^i$ of $X$ admits an epimorphism from a projective $R$-module $P_{X^i}$, and hence $X$ itself admits an epimorphism from the projective object given by
\[ P_X = \cdots \to P_{X^i} \oplus P_{X^{i-1}} \xrightarrow{\big( \begin{smallmatrix} 0 & 0 \\ 1 & 0 \end{smallmatrix} \big)} P_{X^{i+1}} \oplus P_{X^i} \to \cdots. \]
In particular $P_X$ is a contractible complex. As usual, we denote by $\Omega(X)$ the syzygy of $X$, i.e.\ the kernel of the natural projection $P_X \to X$.

Dually, $X$ embeds in a contractible complex $I_X$ consisting of injective modules, and we let $\mho X$ denote the cokernel of $X \hookrightarrow I_X$.

\begin{proposition} \label{prop nice proj resol}
Let $X$ be a complex of $R$-modules. If each term of $X$ has projective dimension no larger than some fixed number $d$, then
\[ \rho(X) = \Omega^d(X)[d]. \]
Dually, if  each term  of $X$ has injective dimension no larger than $d$, then
\[ \lambda(X) = \mho^d(X)[-d]. \]
\end{proposition}

\begin{proof}
Consider the short exact sequence of complexes
\[ 0 \to \Omega(X) \to P_{X} \to X \to 0 \]
with contractible middle term. For any complex $Q$ of projective modules we have
\[ \Ext_R^1(Q^i, \Omega(X)^j) = 0,\]
and hence Lemma~\ref{lem.short_ex} asserts that
\[ \Hom_{\K(\Modules R)}(Q, X) \cong \Hom_{\K(\Modules R)}(Q, \Omega(X)[1]). \]
Iterating we obtain
\[ \Hom_{\K(\Modules R)}(Q, X) \cong \Hom_{\K(\Modules R)}(Q, \Omega^d(X)[d]). \]
But $\Omega^d(X)[d]$ is a complex of projective modules by the assumption on the projective dimensions of the terms of $X$. Thus $\Omega^d(X)[d]$ satisfies the defining isomorphism for the right adjoint, so $\rho(X) = \Omega^d(X)[d]$. %by uniqueness of representing objects.

The proof of the second claim is dual.
\end{proof}

\begin{proposition} \label{prop syzygy-cosyzygy}
Let $X$ and $Y$ be complexes of $R$-modules such that
\begin{equation}
\label{termwiseExtvanishing}
\Ext_{R}^{d+1}(X^m, Y^n) = 0\tag{$\ast$}
\end{equation}
for all terms $X^m$ of $X$ and $Y^n$ of $Y$. Then
\[ \Hom_{\K}(\Omega^d X, Y) \cong \Hom_{\K}(X, \mho^d Y), \]
and this isomorphism is functorial in $X$ and $Y$ satisfying the term-wise vanishing condition $(\ast)$.
\end{proposition}

\begin{proof}
Note that, by dimension shift, for any $0 \leq i \leq d$ we have
\[ \Ext_R^1((\Omega^i X)^m, (\mho^{d-i} Y)^n) = 0.\]
Thus, by Lemma~\ref{lem.short_ex} the short exact sequences
\[
0 \to \mho^{d-i}(Y) \to I_{\mho^{d-i}(Y)} \to \mho^{d-i+1}(Y) \to 0
\]
and
\[
0 \to \Omega^{i+1} X \to P_{\Omega^i X} \to \Omega^i X \to 0
\]
give rise to isomorphisms
\[
\Hom_{\K(\Modules R)} (\Omega^i X, \mho^{d-i} Y) \cong \Hom_{\K(\Modules R)} (\Omega^i X, \mho^{d-i+1} Y[-1])
\]
and
\[
\Hom_{\K(\Modules R)} (\Omega^i X, \mho^{d-i} Y) \cong \Hom_{\K(\Modules R)} (\Omega^{i+1} X[1], \mho^{d-i} Y),
\]
respectively. Combining the latter we find
\[ \Hom_{\K(\Modules R)} (\Omega^i X, \mho^{d-i} Y) \cong \Hom_{\K(\Modules R)} (\Omega^{i-1} X, \mho^{d-i+1} Y), \]
and the claim follows by composing isomorphisms of this form.
\end{proof}

\subsection{Dualizing complexes}
\label{subDualizingcomplexes}
We now introduce Theorem~\ref{thm:iyengarkrause}, due to Iyengar--Krause, which will serve as the foundation for much of the sequel.

\begin{definition} [Iyengar--Krause \cite{MR2262932}]
A \emph{dualizing complex} for a noetherian ring $R$ is a bounded complex $D_{R}$ of $R\text{--}R$-bimodules such that
\begin{enumerate}[label=\emph{(\roman*)}]
\item the terms of $D_{R}$ are injective both as left and as right $R$-modules;
\item the homology of $D_{R}$ is finitely generated both as left and as right $R$-module;
\item the canonical maps $R \to \HOM_{R}(D_{R}, D_{R})$ and $R^{\op} \to \HOM_{R^{\op}}(D_{R}, D_{R})$ are quasi-isomorphisms.
\end{enumerate}
\end{definition}

This terminology is justified by the following result, due to \cite{MR0222093} for commutative rings. A translation of the proof to the non-commutative setting can be found in \cite{MR2262932}.

\begin{theorem} [Hartshorne \cite{MR0222093}] \label{thm:duality}
A dualizing complex $D_{R}$ induces a duality
\[
 \HOM_{R}(-, D_{R}) \colon \Db(\modules R) \leftrightarrow \Db(\modules R^{\op})^{\op} \colon \HOM_{R^{\op}}(-, D_{R}).
\]
\end{theorem}

\begin{example} \label{ex:dualizingforartin}
If $\Lambda$ is an Artin algebra, i.e.\ module finite over a commutative artinian ring $k$, then $D_\Lambda = \Hom_k(\Lambda, E)$ is a dualizing complex for $\Lambda$, where $\Hom_k(-,E)$ is the standard duality.
\end{example}

Recall that a ring is \emph{Gorenstein} if it has finite injective dimension as left and as right module over itself --- see e.g.\ \cite{MR1112170}.

\begin{example} \label{ex over Gorenstein}
If $k$ is a commutative noetherian Gorenstein ring and $\Lambda$ is a module finite (not necessarily commutative) $k$-algebra, then $D_{\Lambda} = \Hom_k(\Lambda, i_k)$ is a dualizing complex for $\Lambda$, where $i_k$ is a finite injective resolution of $k$ over itself.
\end{example}

\begin{theorem}[Iyengar--Krause \cite{MR2262932}] \label{thm:iyengarkrause}
A dualizing complex $D_{\Lambda}$ induces an equivalence
\[ - \otimes_{\Lambda} D_{\Lambda} \colon \K(\Proj \Lambda) \to \K(\Inj \Lambda) \]
with quasi-inverse $\rho \Hom_\Lambda(D_\Lambda,-)$.
\end{theorem}

In the setting of Example~\ref{ex over Gorenstein}, this equivalence interacts nicely with change of rings:

\begin{proposition} \label{prop duality commutes}
Let $f\colon \Lambda\to\Gamma$ be a morphism of noetherian algebras admitting dualizing complexes $D_\Lambda$ and $D_\Gamma$, respectively, such that $\Hom_\Lambda(\Gamma, D_\Lambda) \cong D_\Gamma$ as complexes of $\Lambda$--$\Gamma$-bimodules. Then the equivalences of Theorem~\ref{thm:iyengarkrause} make the following square commutative.
\[ \begin{tikzpicture}
 \node (KIA) at (0,1.5) {$\K(\Inj \Lambda)$};
 \node (KIB) at (5,1.5) {$\K(\Inj \Gamma)$};
 \node (KPA) at (0,0) {$\K(\Proj \Lambda)$};
 \node (KPB) at (5,0) {$\K(\Proj \Gamma)$};
 \draw [<-] (KIA) to node [left] {$\approx$} node [right] {$\scriptstyle -\otimes_\Lambda D_\Lambda$} (KPA);
 \draw [<-] (KIB) to node [left] {$\approx$} node [right] {$\scriptstyle -\otimes_\Gamma D_\Gamma$} (KPB);
 \draw [->] (KIA) to node [above] {$\scriptstyle \Hom_{\Lambda}(\Gamma, -)$} (KIB);
 \draw [->] (KPA) to node [above] {$\scriptstyle - \otimes_{\Lambda} \Gamma$} (KPB);
\end{tikzpicture} \]

In particular, the required isomorphism $\Hom_\Lambda(\Gamma, D_\Lambda) \cong D_\Gamma$ is satisfied by the dualizing complexes $D_{\Lambda}$ and $D_{\Gamma}$ of Example~\ref{ex over Gorenstein} above.
\end{proposition}

\begin{proof}
By assumption, for the first claim it suffices to show that the natural transformation $-\otimes_\Lambda \Hom_\Lambda(\Gamma,D_\Lambda)\to\Hom_\Lambda(\Gamma, -\otimes_\Lambda D_\Lambda)$ is an isomorphism on complexes over $\Proj\Lambda$. This follows readily since the functors coincide on $\Lambda$ and moreover commute with coproducts.

For the last claim, it suffices to observe that \[\Hom_{\Lambda}(\Gamma, D_{\Lambda}) = \Hom_{\Lambda}(\Gamma, \Hom_k(\Lambda, i_k)) = \Hom_k(\Gamma, i_k) = D_{\Gamma}. \qedhere\]
\end{proof}

\subsection{Acyclic, coacyclic and totally acyclic complexes}\label{subsection:(co)acyclicity}

Let $R$ be a ring. We denote by $\K_{\ac}(\Proj R)$ the subcategory of $\K(\Proj R)$ consisting of acyclic complexes. Clearly, $\K_{\ac}(\Proj R)=(\Proj R)^{\perp} =(R)^{\perp}$ as subcategories of $\K(\Proj R)$. Dually, we consider $\K_{\coac}(\Proj R)={}^{\perp}(\Proj R)\subset\K(\Proj R)$ and call an object in this subcategory a \emph{coacyclic complex of projectives}. Note that $X\in\K(\Proj R)$ is coacyclic if and only if the complex $\Hom_R(X,P)$ is exact for each $P \in \Proj R$. Finally, a complex over $\Proj R$ is \emph{totally acyclic} if it belongs to
\[
\Ktac(\Proj R)=\K_{\ac}(\Proj R) \cap \K_{\coac}(\Proj R).
\]

On the other hand, $\K_{\ac}(\Inj{R})$ is the subcategory of $\K(\Inj R)$ whose objects are the acyclic complexes. One easily sees that $\K_{\ac}(\Inj R)={}^{\perp}(\Inj R)= (\lambda R)^{\perp}$ as subcategories of $\K_{\ac}(\Inj R)$ (recall that $\lambda$ denotes an injective resolution). We let $\K_{\coac}(\Inj R)=(\Inj R)^{\perp}\subset\K(\Inj R)$ and call an object of this subcategory a \emph{coacyclic complex of injectives}. Dual to the projective case, $Y\in\K(\Inj R)$ is coacyclic if and only if the complex $\Hom_R(I,Y)$ is exact for each $I\in\Inj R$. Finally, a complex over $\Inj R$ is \emph{totally acyclic} if it belongs to the category
\[
\Ktac(\Inj R)=\K_{\ac}(\Inj R) \cap \K_{\coac}(\Inj R).
\]

\subsubsection*{Right orthogonal classes} It would be convenient if (co)acyclicity was detected by a single object, rather than by all of $\Proj R$ or $\Inj R$. We have already seen how $\K_{\ac}(\Proj R)$ and $\K_{\ac}(\Inj R)$ are obtained as the right orthogonal classes of $R$ and $\lambda R$, respectively. We now show how the presence of a dualizing complex enables us to also write $\K_{\coac}(\Proj R)$ and $\K_{\coac}(\Inj R)$ as $(X)^{\perp}$ for some object $X$. Lemma~\ref{lem:restrictstothick} below is the key ingredient, but to prove this we need the following fact from homological algebra.

\begin{lemma} \label{lem:homofinjectivesisflat}
Let $\Lambda$ be a noetherian ring. Suppose $J\in\Inj\Lambda$ and that $I$ is a $\Lambda\text{--}\Lambda$-bimodule such that ${}_\Lambda I\in\Inj \Lambda^{\op}$. Then $\Hom_\Lambda(I,J)$ is a flat $\Lambda$-module.
\end{lemma}
\begin{proof}
The natural transformation $\Hom_\Lambda(I,J)\otimes_\Lambda- \to \Hom_\Lambda(\Hom_{\Lambda^{\op}}( - ,I),J)$ of right exact functors is an isomorphism on $\Lambda$ and thus on all of $\modules \Lambda^{\op}$. Hence $\Hom_\Lambda(I,J)\otimes_\Lambda-$ is left exact on $\modules \Lambda^{\op}$. This suffices, as flatness of $\Hom_\Lambda(I,J)$ is even implied by the vanishing of $\Tor_1^\Lambda\left(\Hom_{\Lambda}(I,J), \Lambda/L\right)$ for each left ideal $L$.
\end{proof}

In the current subsection we denote by $T$ the equivalence $-\otimes_\Lambda D_\Lambda\colon \K(\Proj \Lambda) \to \K(\Inj \Lambda)$ of Theorem~\ref{thm:iyengarkrause} and by $S$ its quasi-inverse $\rho \Hom_\Lambda(D_\Lambda,-)$. Claim (i) below is essentially \cite[Proposition~4.7]{MR2262932} in the non-commutative setting.

\begin{lemma} \label{lem:restrictstothick}
Let $\Lambda$ be a noetherian ring with a dualizing complex $D_\Lambda$. Then
\begin{enumerate}[label=\emph{(\roman*)}]
\item $T$ restricts to an equivalence $\K^{\bounded}(\Proj \Lambda)\cong \K^{\bounded}(\Inj \Lambda)$;
\item $T$ restricts to an equivalence $\K_{\ac}(\Proj \Lambda) \cong \K_{\coac}(\Inj \Lambda)$;
\item $T$ restricts to an equivalence $\K_{\coac}(\Proj \Lambda) \cong \K_{\ac}(\Inj \Lambda)$;
\item $T$ restricts to an equivalence $\Ktac(\Proj \Lambda) \cong \Ktac(\Inj \Lambda)$.
\end{enumerate}
\end{lemma}
\begin{proof}
Let us start with (i). The functor $T$ does restrict as desired, and we only need to show that so does $S$. For $Y\in\K^{\bounded}(\Inj \Lambda)$, the bounded complex $\Hom_{\Lambda}(D_{\Lambda},Y)$ consists of flat modules by Lemma~\ref{lem:homofinjectivesisflat}. As flat $\Lambda$-modules have finite projective dimension by \cite{MR2236602}, the description of $\rho$ in Proposition \ref{prop nice proj resol} settles the claim. Statement (ii) now follows from
\begin{align*}
   \K_{\ac}(\Proj \Lambda) & = (\Proj \Lambda)^{\perp} = \K^{\bounded}(\Proj \Lambda)^{\perp} \Tcong \K^{\bounded}(\Inj \Lambda)^{\perp} = (\Inj \Lambda)^{\perp} =\K_{\coac}(\Inj \Lambda)
\end{align*}
and a similar argument shows (iii), whence (iv) is immediate.
\end{proof}

\begin{proposition} \label{prop:dualizingtestscoacyclicity}
Let $\Lambda$ be a noetherian ring with a dualizing complex $D_\Lambda$. Then
\begin{enumerate}[label=\emph{(\roman*)}]
\item $\K_{\coac}(\Proj \Lambda) = (\rho \RHom_\Lambda(D_\Lambda, \Lambda))^{\perp}$ in $\K(\Proj\Lambda)$;
\item $\K_{\ac}(\Inj \Lambda) = (\lambda\Lambda)^{\perp}$ in $\K(\Inj\Lambda)$;
\item $\K_{\ac}(\Proj \Lambda) = (\Lambda)^{\perp}$ in $\K(\Proj\Lambda)$;
\item $\K_{\coac}(\Inj \Lambda) = (D_\Lambda)^{\perp}$ in $\K(\Inj\Lambda)$.
\end{enumerate}

Moreover, in \textnormal{(i)}--\textnormal{(iv)} the objects defining the right orthogonal classes are all compact in $\K(\Proj \Lambda)$ and $\K(\Inj \Lambda)$, respectively.
\end{proposition}
\begin{proof}
Claims (ii) and (iii) were observed above. We obtain (i) from the following identifications
\[
\K_{\coac}(\Proj \Lambda) \Tcong \K_{\ac}(\Inj \Lambda) = (\lambda \Lambda)^{\perp} \Scong (\rho \Hom_\Lambda(D_\Lambda,\lambda\Lambda))^{\perp} = (\rho \RHom_\Lambda(D_\Lambda, \Lambda))^{\perp}. \]
For (iv), Lemma~\ref{lem:restrictstothick} implies $\K^{\bounded}(\Inj \Lambda)=\thick(\Add D_\Lambda)$ from which we infer
\[
\K_{\coac}(\Inj \Lambda) = (\Inj \Lambda)^{\perp} = \K^{\bounded}(\Inj \Lambda)^{\perp} = \thick(\Add D_\Lambda)^{\perp} = (D_\Lambda)^{\perp}.
\]

Finally, since $\lambda\Lambda$ is compact in $\K(\Inj \Lambda)$ it follows that $S(\lambda\Lambda)=\rho\RHom_\Lambda(D_\Lambda, \Lambda)$ is compact in $\K(\Proj \Lambda)$. Similarly, the compactness of $\Lambda$ in $\K(\Proj \Lambda)$ implies the compactness of $T(\Lambda)=D_\Lambda$ in $\K(\Inj \Lambda)$.
\end{proof}

\subsubsection*{Left orthogonal classes} Contrary to what one might expect, $\K_{\coac}(\Proj \Lambda)$ need not coincide with ${}^{\perp}(\Lambda)$ even for a noetherian ring $\Lambda$. In fact, exactness of $\Hom_{\Lambda}(X,\Lambda)$ fails to imply coacyclicity of a complex $X$ over $\Proj \Lambda$ already if $\Lambda$ is a complete local domain \cite[Remark~5.11]{MR2262932}. Knowing this, the below Proposition~\ref{prop:leftperpforartin} is more or less what one could hope for in the pursuit of describing (co)acyclicity as ${}^{\perp}(X)$ for a single object $X$. We will need the following observation, the proof of which uses the machinery of pure-injective modules and can be found in the Appendix.

\begin{lemma} \label{lem:naturalmonosplitartin}
If $\Lambda$ is an Artin algebra, then the natural monomorphism $\Lambda^{(I)} \to \Lambda^I$ is split for any index set $I$.
\end{lemma}

Below, $D_\Lambda$ refers to the dualizing complex from Example~\ref{ex:dualizingforartin}.

\begin{proposition} \label{prop:leftperpforartin}
Let $\Lambda$ be an Artin algebra. Then
\begin{enumerate}[label=\emph{(\roman*)}]
\item $\K_{\coac}(\Proj \Lambda) = {}^{\perp}(\Lambda)$ in $\K(\Proj\Lambda)$;
\item $\K_{\ac}(\Inj \Lambda) = {}^{\perp}(D_\Lambda)$ in $\K(\Inj\Lambda)$;
\item $\K_{\ac}(\Proj \Lambda) = {}^{\perp}(\rho D_\Lambda)$ in $\K(\Proj\Lambda)$;
\item $\K_{\coac}(\Inj \Lambda) = {}^{\perp}(D_\Lambda \otimes^{\mathbb{L}}_\Lambda D_\Lambda)$ in $\K(\Inj\Lambda)$.
\end{enumerate}
\end{proposition}
\begin{proof}
To prove (i) observe that Lemma~\ref{lem:naturalmonosplitartin} implies $\Proj \Lambda = \Prod \Lambda$, which yields $\K_{\coac}(\Proj \Lambda)={}^{\perp}(\Proj \Lambda)={}^{\perp}(\Prod \Lambda)={}^{\perp}(\Lambda)$. Now (ii) follows once we identify
\[
\K_{\ac}(\Inj \Lambda)\Scong\K_{\coac}(\Proj \Lambda) = {}^{\perp}(\Lambda) \Tcong {}^{\perp}(D_\Lambda).
\]
For (iii), notice that we also have $\Inj \Lambda = \Prod D_\Lambda$ for our dualizing complex. Moreover, a complex $X$ is acyclic if and only if $\Hom_\Lambda(X,I)$ is exact for each $I\in\Inj \Lambda$. Hence in $\K(\Modules \Lambda)$ we have
\[
\K_{\ac}(\Proj \Lambda) = {}^{\perp}(\Inj \Lambda) \cap \K(\Proj \Lambda) = {}^{\perp}(\Prod D_\Lambda) \cap \K(\Proj \Lambda) = {}^{\perp}(D_\Lambda) \cap \K(\Proj \Lambda).
\]
Observe next that by adjunction we have $\Hom_{\K}(P, \rho D_\Lambda) = \Hom_{\K}(P, D_\Lambda)$ for each complex $P$ over $\Proj \Lambda$, which reveals that $\K_{\ac}(\Proj \Lambda)={}^{\perp}(\rho D_\Lambda)$ in $\K(\Proj \Lambda)$. Now (iv) is implied by identifying
\[
\K_{\coac}(\Inj \Lambda) \Scong \K_{\ac}(\Proj \Lambda) = {}^{\perp}(\rho D_\Lambda) \Tcong {}^{\perp}(\rho D_\Lambda \otimes_\Lambda D_\Lambda) = {}^{\perp}(D_\Lambda \otimes^{\mathbb L}_\Lambda D_\Lambda). \qedhere
\]
\end{proof}

\begin{remark} \label{remark:fingencoacyclic}
For an additive category $\A$ one lets $\Ktac(\A)={}^{\perp}(\A) \cap (\A)^{\perp}$ upon viewing $\A$ as the stalk complexes of $\K(\A)$. In this tradition, if $R$ is a ring, then $\K_{\coac}(\proj R)$ should be defined as $\K(\proj R) \cap {}^{\perp}(\proj R)$. On the other hand, it would also be natural to define $\K_{\coac}(\proj R)$ as $\K(\proj R) \cap \K_{\coac}(\Proj R)$. Thankfully, this seeming conflict solves itself. Indeed, the fact that each finitely generated module is compact in $\Modules R$ implies that ${}^{\perp}(\proj R) = {}^{\perp}(\Proj R)$ as subcategories of $\K(\proj R)$, so the competing definitions agree. In particular, and useful in Section~\ref{section:singularitycategories} below, this means that a complex $P$ over $\proj R$ does belong to $\K_{\coac}(\proj R)$ if $\Hom_{\K}(P,R)=0$.
\end{remark}

\subsection{Gorenstein homological algebra}
\label{subGorHomAlg}
Let $R$ be a ring. An $R$-module is \emph{Gorenstein projective} if it appears as the $0$-boundaries of a totally acyclic complex over $\Proj R$. The Gorenstein projective $R$-modules form a Frobenius exact subcategory $\GProj R$ of $\Modules R$, and assigning $X \mapsto \B^0(X)$ gives a triangle equivalence
\[
\Ktac(\Proj R)\cong\GStable R.
\]
Provided $R$ is right noetherian, $\Gproj R = \GProj R \cap \modules R$ is again Frobenius exact in $\modules R$, and
\[
\Ktac(\proj R)\cong\Gstable R
\]
by restriction of the above equivalence. Dually, an $R$-module is called \emph{Gorenstein injective} if it is isomorphic to the $0$-cycles of some totally acyclic complex over $\Inj R$, and assigning $Y \mapsto \B^0(Y)$ gives triangle equivalences
\[
\Ktac(\Inj R)\cong\GcoStable R\text { \,and } \Ktac(\inj R)\cong\Gcostable R.
\]
Notice that if $R$ happens to be noetherian with a dualizing complex, then
\[
\GStable R \cong \GcoStable R \text{ \,and } \Gstable R \cong \Gcostable R
\]
by Lemma~\ref{lem:restrictstothick}.

If $\Gamma$ is an Artin algebra, then the duality between $\Gproj \Gamma$ and $\Ginj \Gamma$ is pleasant enough to ensure $\left(\Gproj \Gamma\right)^{\perp}= {}^{\perp}\!\left(\Ginj \Gamma\right)$ in $\modules \Gamma$ (see \cite{MR2138374}). However, this does not necessarily hold true for big modules. To amend this oddity, in \cite{MR2327478} the class of \emph{virtually Gorenstein} algebras was introduced as the algebras $\Gamma$ for which $\left(\GProj \Gamma \right)^{\perp}={}^{\perp}\!\left(\GInj \Gamma\right)$. We remark that the class of virtually Gorenstein algebras is rather large. Indeed, it contains the algebras of finite representation type and is closed under derived equivalence. In fact it seems that the first example of an Artin algebra which is not virtually Gorenstein appeared as recently as \cite[Example~4.3]{MR2524651}.

Denote now by $\Lambda$ a noetherian ring. The \emph{singularity category} of $\Lambda$ is the Verdier quotient
\[
\Dsg(\Lambda)=\Db(\modules \Lambda)/\!\perf \Lambda
\]
introduced in \cite{buchweitz1986maximal}. Notice that there is a natural triangle functor
\[
\iota \colon \Gstable\Lambda\to \Dsg(\Lambda).
\]
Indeed, upon identifying $\Gstable \Lambda=\Ktac(\proj \Lambda)$ and $\Db(\modules \Lambda)=\K^{-, \bounded}(\proj \Lambda)$, $\iota$ is given by assigning to a totally acyclic complex $X$ over $\proj \Lambda$ the object in $\Dsg(\Lambda)$ represented by its brutal truncation $\tau^0 X$.

Below, (i) is the fundamental theorem of Buchweitz \cite[Theorem~4.4.1]{buchweitz1986maximal}. The partial converses in (ii) are due to \cite{MR1780017,MR3356832} (see also \cite{MR2790881} for a relative version).

\begin{theorem} \label{thm:embeddingofgproj}
The functor $\iota$ is fully faithful. Moreover,
\begin{enumerate}[label=\emph{(\roman*)}]
\item if $\Lambda$ is Gorenstein, then $\iota$ gives an equivalence $\Gstable \Lambda \cong\Dsg(\Lambda)$;
\item if $\Lambda$ is commutative local or an Artin algebra such that $\iota$ is an equivalence, then $\Lambda$ is Gorenstein.
\end{enumerate}
\end{theorem}

\subsection{Compactly generated completions} \label{subsection:compactcompletions}
When $\Lambda$ is a noetherian ring, the triangulated categories $\Db(\modules \Lambda)$ and $\Dsg(\Lambda)$ may be realized as the subcategories of compact objects in familiar compactly generated triangulated categories. If $\Lambda$ is moreover virtually Gorenstein, then the same goes for $\Gstable \Lambda$. We now explain how these embeddings come about.

\begin{theorem}\label{thm:compactcompletions}
If $\Lambda$ is a noetherian ring, then we have the following.
\begin{enumerate}[label=\emph{(\roman*)}]
\item $\K(\Inj \Lambda)$ is compactly generated with $\K(\Inj \Lambda)^{\cc} \cong \Db(\modules \Lambda)$;
\item $\K_{\ac}(\Inj \Lambda)$ is compactly generated with $\K_{\ac}(\Inj \Lambda)^{\cc} \cong \Dsg(\Lambda)$;
\item $\Ktac(\Inj \Lambda)$ is compactly generated if $\Lambda$ admits a dualizing complex. If moreover $\Lambda$ is an Artin algebra, then $\Ktac(\Inj \Lambda)^{\cc}\cong\Gstable \Lambda$ if and only if $\Lambda$ is virtually Gorenstein.
\end{enumerate}
\end{theorem}

\begin{proof}[Idea of proof]
Claims (i) and (ii) are due to Krause \cite{MR2157133} --- let us give a brief account of his argument. It is straightforward to verify that the injective resolution of a finitely generated module is compact in $\K(\Inj \Lambda)$, from which it follows that $\K(\Inj \Lambda)$ is compactly generated. In other words, taking injective resolutions embeds $\Db(\modules \Lambda)$ as the compact objects of $\K(\Inj \Lambda)$.

Recall that $\K_{\ac}(\Inj\Lambda)$ is the subcategory $(\lambda\Lambda)^{\perp}$ of $\K(\Inj \Lambda)$ by Proposition~\ref{prop:dualizingtestscoacyclicity}, so in particular it is closed under coproducts. As the compact generation of $\K(\Inj\Lambda)$ has already been established, we may invoke \cite{MR1191736} which constructs a compact preserving left adjoint $I_\lambda$ to the inclusion $I\colon\K_{\ac}(\Inj\Lambda)\to\K(\Inj\Lambda)$. As such an $I_\lambda$ automatically takes a set of generating objects to a set of generating objects, the compact generation of $\K_{\ac}(\Inj \Lambda)$ follows. Now Theorem~\ref{thm:neemanthreeadjoints} yields the existence of a recollement
\[
\K_{\ac}(\Inj \Lambda) \hookrightarrow \K(\Inj \Lambda) \to  \D(\Modules \Lambda)
\]
%\begin{center}
%   \begin{tikzpicture}
%     \matrix(a)[matrix of math nodes,
%     row sep=2.5em, column sep=4em,
%     text height=1.5ex, text depth=0.25ex]
%     { \K_{\ac}(\Inj \Lambda) & \K(\Inj \Lambda) &  \D(\Modules \Lambda) \\};
%     \path[->,font=\scriptsize](a-1-1) edge node[above=-.5mm]{$I$} (a-1-2);
%    \path[->,bend left, font=\scriptsize](a-1-2) edge (a-1-1);
%     \path[->,bend right, font=\scriptsize](a-1-2) edge node[above=-.75mm]{$I_\lambda$} (a-1-1);
%     \path[->,font=\scriptsize](a-1-2) edge (a-1-3);
%     \path[->,bend left, font=\scriptsize](a-1-3) edge (a-1-2);
%     \path[->,bend right, font=\scriptsize](a-1-3) edge (a-1-2);
%\end{tikzpicture}
%\end{center}
which induces an equivalence up to direct summands
\[
\K_{\ac}(\Inj \Lambda)^{\cc}\cong\K(\Inj \Lambda)^{\cc}/\D(\Modules \Lambda)^{\cc},
\]
and the fact that $\D(\Modules \Lambda)^{\cc} = \perf \Lambda$ is well known.

We now turn to (iii). As in the previous paragraph, the compact generation of $\Ktac(\Inj \Lambda)$ follows from the compact generation of $\K_{\ac}(\Inj\Lambda)$ since $\Ktac(\Inj \Lambda)$ is the subcategory $(D_\Lambda)^{\perp}$ of $\K_{\ac}(\Inj\Lambda)$ by Proposition~\ref{prop:dualizingtestscoacyclicity}. For the last claim, $\Gstable \Lambda \subset \left(\GStable \Lambda\right)^{\cc}$ is always true. By \cite{MR2737805}, the reversed inclusion holds precisely when $\Lambda$ is virtually Gorenstein, which is not entirely surprising once one learns that $\Lambda$ is virtually Gorenstein if and only if each Gorenstein projective module is a filtered colimit of finitely generated Gorenstein projective modules \cite[Theorem~5]{MR2524651}. Hence the proof is complete since $\GStable \Lambda \cong \Ktac(\Inj \Lambda)$.
\end{proof}

In the above, the homotopy categories appearing are over $\Inj \Lambda$, even though Lemma~\ref{lem:restrictstothick} suggests that we could equally well work over $\Proj \Lambda$. Indeed, each relevant property will formally carry over by transport of structure, but homotopy categories over $\Inj \Lambda$ seem to be intrinsically better behaved than their counterparts over $\Proj \Lambda$. For instance, and crucial above, if $M$ is a finitely generated $\Lambda$-module, then $\lambda M$ is compact in $\K(\Inj \Lambda)$ while $\rho M$ need not be compact in $\K(\Proj \Lambda)$. Indeed, $\rho M$ is compact in $\K(\Proj\Lambda)$ if and only if $M$ is compact in $\D(\Modules\Lambda)$, i.e.\ precisely when $M$ has finite projective dimension.

A further reason for preferring injectives is that generalizations are then more often within reach. Indeed, in the above (i) and (ii) one may replace $\Modules \Lambda$ by any locally noetherian Grothendieck category $\mathcal A$ such that $\D(\mathcal A)$ is compactly generated, which happens for instance when $\mathcal A$ has finite global dimension. On the other hand, there is no reason why such an $\mathcal A$ should even have enough projectives. Nevertheless, Jørgensen \cite{MR2132765} showed directly, i.e.\ with no allusion to the above, that $\K(\Proj \Lambda)$ is compactly generated if $\Lambda$ is noetherian. Funnily enough, the compact objects of $\K(\Proj \Lambda)$ arise as $\Hom_{\Lambda^{\op}}(\rho M,\Lambda)$ for $M\in \modules \Lambda^{\op}$, and it turns out that $\K(\Proj \Lambda)^{\cc}$ is naturally equivalent to $\Db(\modules \Lambda^{\op})^{\op}$. Note that the latter should be expected in light of the Grothendieck-type duality of Theorem~\ref{thm:duality}.

For later reference we end with a porism of Theorem~\ref{thm:compactcompletions}. Below, the statements involving $\Proj \Lambda$ follow from those involving $\Inj\Lambda$ by restricting the equivalence $\K(\Inj\Lambda)\cong\K(\Proj\Lambda)$.

\begin{observation} \label{obs:leftandrightadjoints}
Let $\Lambda$ be a noetherian ring admitting a dualizing complex. Then
\begin{enumerate}[label=\emph{(\roman*)}]
\item the inclusions $\K_{\ac}(\Inj\Lambda)\hookrightarrow\K(\Inj\Lambda)$ and $\K_{\coac}(\Proj\Lambda)\hookrightarrow\K(\Proj\Lambda)$ both admit a left and a right adjoint;
\item the inclusions $\Ktac(\Inj \Lambda)\hookrightarrow\K_{\ac}(\Inj\Lambda)$ and $\Ktac(\Proj \Lambda)\hookrightarrow\K_{\coac}(\Proj\Lambda)$ both admit a left and a right adjoint.
\end{enumerate}
\end{observation}

\subsection{Contravariant finiteness and torsion pairs}
\label{lastsubsection}
Let $\A$ be a category with a subcategory $\B$. A \emph{right $\B$-approximation} of $A\in\A$ is a morphism $B\to A$ with $B\in\B$ through which each morphism $B'\to A$ with $B'\in\B$ factors. $\B$ is called \emph{contravariantly finite} in $\A$ if each $A\in\A$ admits a right $\B$-approximation. The dual notions are those of a \emph{left approximation} and a \emph{covariantly finite subcategory}, respectively. In the following sense, it is easy to generate such categories.

\begin{lemma} \label{lem:covariantlyfinite}
Let $\A$ be a category and $\B$ a skeletally small subcategory. If $\A$ has coproducts, then $\Add\B$ is contravariantly finite in $\A$. Dually, if $\A$ has products, then $\Prod\B$ is covariantly finite in $\A$.
\end{lemma}
\begin{proof}
   Assume that $\A$ has coproducts. For $A \in \A$, let $I$ denote the collection of all morphisms $B_i\to A$ with $B_i\in\B$. Then $I$ is a set and the canonical morphism
   \[
   \coprod_{i\in I} B_i \to A
   \]
   is a right $\Add \B$-approximation. The remaining claim is dual.
\end{proof}

A basic problem in Gorenstein homological algebra is determining when $\GProj R$ is contravariantly finite in $\Modules R$. For some time, an affirmative answer could only be given under rather strong restrictions --- see \cite[Theorem~2.9]{MR1355071} and \cite[Theorem~3.4]{MR1432346}. Jørgensen vastly improved on these when he showed in \cite{MR1859047} that any module over an Artin algebra admits a right approximation by Gorenstein projective modules, and later used similar techniques in order to extend his theorem to the one below. We remark that one could hope to go even further, as there seems to be no known example of a ring $R$ such that $\GProj R$ fails to be contravariantly finite in $\Modules R$.

\begin{theorem}[Jørgensen \cite{MR2283103}] \label{thm:jorgensenGProjcontfinite}
Consider either of the following two situations.
\begin{enumerate}[label=\emph{(\roman*)}]
  \item $\Lambda$ is a commutative noetherian ring admitting a dualizing complex.
  \item $\Lambda$ is a left coherent, right noetherian algebra over a field $k$ for which there exists a left noetherian $k$-algebra $\Gamma$ and a dualizing complex ${}_{\Gamma} D_\Lambda$ in the sense of \textnormal{\cite{MR1674648}}.
\end{enumerate}
Then $\GProj \Lambda$ is contravariantly finite in $\Modules \Lambda$.
\end{theorem}

A pair of subcategories $(\X,\Y)$ of a triangulated catgory $\T$ is a \emph{torsion pair} if $\T(\X,\Y)=0$, and each object $T\in\T$ appears in a triangle
\[
X_T \to T \to Y_T \to  X_T[1]
\]
with $X_T \in \X$ and $Y_T \in \Y$. If in addition we have $\X[1]\subseteq \X$, the torsion pair $(\X,\Y)$ is called a \emph{$t$-structure}. In this case $\X$ is an \emph{aisle} and $\Y$ a \emph{coaisle} in $\T$. We remark that torsion pairs according to \cite[Definition I.2.1]{MR2327478} are in fact precisely t-structures. 
%We remark that in the literature, a torsion pair is more or less interchangeable with a \emph{$t$-structure}, as assigning $(\X,\Y) \mapsto (\X,\Y[1])$ gives a bijection between the torsion pairs and the $t$-structures in $\T$. 
An aisle is always contravariantly finite since the above $X_T\to T$ is a right $\X$-approximation of $T$. Moreover, approximations coming from aisles are functorial, since assigning $T\mapsto X_T$ gives a right adjoint $\R$ to the inclusion $\X\hookrightarrow\T$ --- originally due to Keller--Vossieck \cite{MR907948}. As the dual claims hold for coaisles, the following picture always accompanies a t-structure.
\begin{center}
   \begin{tikzpicture}
      \matrix(a)[matrix of math nodes,
      row sep=2.5em, column sep=3em,
      text height=1.5ex, text depth=0.25ex]
      {  \X & \T & \Y \\};
      \path[right hook->,font=\scriptsize](a-1-1) edge (a-1-2);
      \path[left hook ->,font=\scriptsize](a-1-3) edge (a-1-2);
      \path[->,bend left, font=\scriptsize](a-1-2) edge node[below]{$\RR$} (a-1-1);
      \path[->,bend left, font=\scriptsize](a-1-2) edge node[above]{$\LL$} (a-1-3);
   \end{tikzpicture}
\end{center}
This `decomposition' of the ambient category $\T$ means that an aisle is a much stronger tool than a contravariantly finite subcategory. It is correspondingly more difficult to get a hold of and, indeed, the matter of generating aisles has become a central issue in modern algebra. Let us mention one remarkable result.

\begin{theorem} [{{Neeman \cite{MR2680406}}}]
Let $\T$ be an idempotent complete triangulated category, and $\SSS$ a thick contravariantly finite subcategory. Then $\SSS$ is an aisle in $\T$.
\end{theorem}

A stronger notion still is that of a \emph{stable $t$-structure} on $\T$, i.e.\ a t-structure $(\X,\Y)$ in which $\X$ and $\Y$ are both closed under all shifts. In this case $\X$ and $\Y$ become thick subcategories of $\T$, and the above adjoints $\RR$ and $\LL$ induce triangle equivalences $\T/\Y \cong \X$ and $\T/\X \cong \Y$, respectively.

\section{Singularity categories and Gorenstein projectives} \label{section:singularitycategories}

In this section we are concerned with how change of rings affects `small' categories, that is categories derived in some way from categories of finitely generated modules. Therefore, throughout this section, we assume $f \colon \Lambda \to \Gamma$ to be a morphism of noetherian algebras such that $\pdim{}_{\Lambda}\Gamma$ and $\pdim \Gamma_\Lambda$ are finite.

The following theorem sums up the primary result of this section.

\begin{theorem} \label{thm finite main}
Under the standing assumptions above we have the solid part of the following commutative diagram
\[ \begin{tikzpicture}
    \matrix(a)[matrix of math nodes,row sep=2.5em, column sep=4em,text height=1.5ex, text depth=0.25ex]
      { \Gstable \Lambda & \Gstable \Gamma  \\
      \Dsg(\Lambda) & \Dsg(\Gamma) \\ };
      \draw[right hook->,shorten <=.3em] (a-1-1) to (a-2-1);
      \draw[right hook->,shorten <=.3em] (a-1-2) to (a-2-2);
      \draw[->,bend left=30] (a-1-1) to node [above] {\scriptsize $- \otimes_{\Lambda} \Gamma$} (a-1-2);
      \draw[->, dashed] (a-1-2) to node[above]{\scriptsize $\res$} (a-1-1);
      \draw[->,bend left=30] (a-2-1) to node [above] {\scriptsize $- \otimes^{\mathbb{L}}_{\Lambda} \Gamma$} (a-2-2);
      \draw[->] (a-2-2) to  node [above] {\scriptsize $\res$} (a-2-1);
      \draw[->, dashed, bend right=30] (a-2-1) to node [below] {\scriptsize $\RHom_{\Lambda}(\Gamma, -)$} (a-2-2);
   \end{tikzpicture} \]
where the two functors on the bottom level form an adjoint pair.

If moreover $\RHom_{\Lambda}(\Gamma, \Lambda) \in \perf \Gamma$, then $\res$ restricts to a functor between stable categories of finitely generated Gorenstein projective modules, and it has a right adjoint $\RHom_{\Lambda}(\Gamma, -)$ on the level of singularity categories, as indicated by the dashed arrows above.
\end{theorem}

\subsection*{Singularity categories} Our strategy here is to observe first that the desired functors exist between bounded derived categories, and then transfer them to singularity categories.

\begin{lemma}
Under the assumptions at the beginning of this section, there is an adjoint triple
\[ \begin{tikzpicture}
      \matrix(a)[matrix of math nodes,
      row sep=2.5em, column sep=4em,
      text height=1.5ex, text depth=0.25ex]
      { \Db(\modules \Lambda) & \Db(\modules \Gamma).  \\};
      \path[->,font=\scriptsize](a-1-2) edge node[above]{$\res$} (a-1-1);
      \path[->,bend left, font=\scriptsize](a-1-1) edge node[above]{$- \otimes^{\mathbb{L}}_{\Lambda} \Gamma$} (a-1-2);
      \path[->,bend right, font=\scriptsize](a-1-1) edge node[below]{$\RHom_{\Lambda}(\Gamma, -)$} (a-1-2);
\end{tikzpicture} \]
\end{lemma}
\begin{proof}
Deriving the initial adjoint triple of functors between the big module categories gives an adjoint triple between $\D(\Modules \Lambda)$ and $\D(\Modules \Gamma)$. It is straightforward to verify that the latter restricts to bounded derived categories in our setup.
\end{proof}

This diagram immediately gives us the lower part of Theorem~\ref{thm finite main}:

\begin{proposition}
There is an adjoint pair of functors between singularity categories as in the following diagram.
\[ \begin{tikzpicture}
    \matrix(a)[matrix of math nodes,row sep=2.5em, column sep=4em,text height=1.5ex, text depth=0.25ex]
      { \Dsg(\Lambda) & \Dsg(\Gamma) \\ };
       \draw[->,bend left=30] (a-1-1) to node [above] {\scriptsize $- \otimes^{\mathbb{L}}_{\Lambda} \Gamma$} (a-1-2);
      \draw[->] (a-1-2) to  node [above] {\scriptsize $\res$} (a-1-1);
      \draw[->, dashed, bend right=30] (a-1-1) to node [below] {\scriptsize $\RHom_{\Lambda}(\Gamma, -)$} (a-1-2);
   \end{tikzpicture} \]
If moreover $\RHom_{\Lambda}(\Gamma, \Lambda) \in \perf \Gamma$, then there is an adjoint triple as indicated by the dashed arrow.
\end{proposition}

\begin{proof}
By Lemma~\ref{lemadjointquotient} it suffices to check that the respective functors between bounded derived categories preserve perfect complexes. For $- \otimes^{\mathbb{L}}_{\Lambda} \Gamma$ this is automatic, and for $\res$ it follows since $\Gamma$ is quasi-isomorphic to a perfect complex of $\Lambda$-modules by assumption. Finally, $\RHom_{\Lambda}(\Gamma, -)$ preserves perfect complexes if and only if $\RHom_{\Lambda}(\Gamma, \Lambda) \in \perf \Gamma$.
\end{proof}

\subsection*{Stable categories of Gorenstein projective modules} Recall that in the current section, the coacyclicity of a complex of projective modules is detected by the ring itself --- see Remark~\ref{remark:fingencoacyclic}.

\begin{proposition}
If $M\in\Gproj \Lambda$, then $M \otimes_{\Lambda}\Gamma$ is a Gorenstein projective $\Gamma$-module and $\Tor^{\Lambda}_n(M, \Gamma) = 0$ for $n > 0$.
\end{proposition}

\begin{proof}
By definition, there is a totally acyclic complex $P$ over $\proj \Lambda$ such that $M = \B^0(P)$. As $P$ is exact we see that $\HH^i(P \otimes_{\Lambda} \Gamma) = \Tor_{j-i}(\B^{j+1}(P), \Gamma)$ for any $j > i$. Since the projective dimension of $\Gamma$ as left $\Lambda$-module is assumed to be finite, we can always choose $j$ sufficiently big so that this $\Tor$ vanishes. Thus $P \otimes_{\Lambda} \Gamma$ is exact.

It follows that $\Tor^{\Lambda}_n(M, \Gamma) = 0$, and that $M \otimes_{\Lambda} \Gamma = \B^0( P \otimes_{\Lambda} \Gamma)$. Thus the proof is complete provided we manage to show that $P\otimes_\Lambda \Gamma$ is totally acyclic. Acyclicity is already established. For coacyclicity note that
\[ \Hom_{\Lambda}( P \otimes_{\Lambda} \Gamma, \Gamma) = \Hom_{\Lambda}(P, \Gamma), \]
and that the latter is exact because $\Gamma \in \perf \Lambda$ as right $\Lambda$-modules, and  $\Hom_{\Lambda}(P, \Lambda)$ is exact by assumption.
\end{proof}

\begin{proposition} \label{lem:resgivesfunctorbetweengorenstein}
Assume that $\RHom_{\Lambda}(\Gamma, \Lambda) \in \perf \Gamma$. Then the restriction functor between the singularity categories restricts to a functor between the stable categories of finitely generated Gorenstein projective modules.
\end{proposition}

For the proof, we prepare the following technical observation.

\begin{lemma} \label{lem.rho stably same}
Let $X$ be an acyclic complex of finitely generated $\Lambda$-modules, and suppose the projective dimensions of the terms of $X$ are bounded by $d$. Then, for each $i$, there is a morphism $\phi \colon \B^i(\rho X)\to \B^i(X)$ such that the cone of $\phi$ is quasi-isomorphic to a complex of finitely generated projectives concentrated in degrees $-d$ to $0$.
\end{lemma}

\begin{proof}
By Proposition~\ref{prop nice proj resol} we have $\rho X = \Omega^d X[d]$. It follows directly from the construction of syzygies for complexes that syzygies commute with taking boundaries of exact complexes, thus
\[ \B^i(\rho X) = \B^i(\Omega^d X[d]) = \Omega^d \B^{i+d}(X). \]
We hence obtain a commutative diagram
\[
\begin{tikzpicture}[baseline=(current bounding box.center)]
   \matrix(a)[matrix of math nodes,
   row sep=3em, column sep=2em,
   text height=1.5ex, text depth=0.25ex]
   {\B^i(\rho X) & P_{d-1} & \cdots & P_0 & \B^{i+d}(X) \\
   \B^i(X) & X^i & \cdots & X^{i+d-1} & \B^{i+d}(X), \\};
   \path[->,font=\scriptsize](a-1-1) edge (a-1-2)
   (a-1-2) edge (a-1-3)
   (a-1-3) edge (a-1-4)
   (a-1-4) edge (a-1-5)

   (a-2-1) edge (a-2-2)
   (a-2-2) edge (a-2-3)
   (a-2-3) edge (a-2-4)
   (a-2-4) edge (a-2-5)

   (a-1-1) edge node [left] {$\phi$} (a-2-1)
   (a-1-2) edge (a-2-2)
   (a-1-4) edge (a-2-4);
   \draw[double distance=.5mm](a-1-5) to (a-2-5);
\end{tikzpicture}
\]
from which it follows that there is a quasi-isomorphism
\[ \Cone(\phi) \to \left[P_{d-1} \to P_{d-2} \oplus X^i \to \cdots \to \B^{i+d}(X) \oplus X^{i+d-1} \to \B^{i+d}(X) \right]. \]
Note that the rightmost map here is a split epimorphism. After splitting off the term $ \B^{i+d}(X)$, the complex
\[ P_{d-1} \to P_{d-2} \oplus X^i \to \cdots \to P_0 \oplus X^{i+d-2} \to X^{i+d-1}  \]
extends from degrees $-1$ to $d-1$. Moreover, this complex admits non-zero homology only in degrees $-1$ and $0$. In particular it is quasi-isomorphic to the complex
\[ P_{d-1} \to K, \]
where $K$ is the kernel of the map $P_{d-2} \oplus X^i \to P_{d-3} \oplus X^{i+1}$. Since $K$ appears as the kernel of the exact sequence
\[ 0 \to K \to P_{d-2} \oplus X^i \to \cdots \to P_0 \oplus X^{i+d-2} \to X^{i+d-1} \to 0 \]
in which all other terms have projective dimension at most $d$, it follows that also the projective dimension of $K$ is bounded above by $d$. Thus, replacing $K$ by a minimal projective resolution we obtain the claim.
\end{proof}

\begin{proof}[Proof of Proposition~\ref{lem:resgivesfunctorbetweengorenstein}]
Let $M\in\Gproj\Gamma$. What we need to show is that $\res M$ is, up to isomorphism in $\Dsg(\Lambda)$, a Gorenstein projective $\Lambda$-module.

Let $P$ be a complete resolution of $M$, that is a totally acyclic complex over $\proj \Gamma$ such that $M = \B^0(P)$. Since finitely generated projective $\Gamma$-modules have bounded projective dimension as $\Lambda$-modules we are in the situation of Lemma~\ref{lem.rho stably same} above. In particular
\[ \res M = \B^0(\res P) \cong \B^0(\rho \res P), \]
since $\B^0(\res P)$ and $\B^0(\rho \res P)$ only differ by a finite complex of projectives.

It only remains to show that $\B^0(\rho \res P)$ is Gorenstein projective over $\Lambda$. For this it clearly suffices to check that $\rho \res P$ is totally acyclic. The fact that $\rho \res P$ is acyclic follows directly from the construction. Thus it only remains for us to show that $\Hom_{\Lambda}( \rho \res P, \Lambda)$ is also acyclic. One may observe that the homologies of this complex are, for sufficently large $j$, given by
\begin{align*}
  \HH^i(\Hom_{\Lambda}( \rho \res P, \Lambda)) & = \Ext^{i + j}_{\Lambda}\left( \B^{j+1}(\rho \res P), \Lambda\right) \\
& = \Ext^{i+j}_{\Lambda}\left(\B^{j+1}(\res P), \Lambda\right) & \text{Lemma~\ref{lem.rho stably same} for $i+j>d$} \\
& = \Ext^{i+j}_{\Lambda}\left(\res \B^{j+1}(P), \Lambda\right) \\
& = \Ext^{i+j}_{\Gamma}\left(\B^{j+1}(P), \RHom_\Lambda(\Gamma,\Lambda)\right) 
\end{align*}
which vanishes provided $\RHom_{\Lambda}(\Gamma, \Lambda)$ is perfect, since $P$ is totally acyclic.
\end{proof}

This furnishes the proof of Theorem~\ref{thm finite main}.

\subsection*{Applications}

For the remainder of the current section, we assume that $\Lambda$ and $\Gamma$ are module finite over some regular commutative noetherian ring $k$, and that $f$ is $k$-linear. Our next result can be interpreted as saying that $\Lambda$ is `less singular' than $\Gamma$ under the following condition on the cone of $f$.

\begin{proposition} \label{prop:coneperfectimpliesfullyfaithful}
Suppose $\Cone(f)\in\perf\Lambdae$. Then, in the notation of Theorem~\ref{thm finite main},
\begin{enumerate}[label=\emph{(\roman*)}]
\item the functor $-\otimes_{\Lambda}^{\mathbb L}\Gamma\colon\Dsg(\Lambda)\to\Dsg(\Gamma)$ is fully faithful;

\item the functor $-\otimes_{\Lambda}\Gamma\colon\Gstable \Lambda \to \Gstable \Gamma $ is fully faithful.
\end{enumerate}
\end{proposition}
\begin{proof}
For (i) it suffices to show that the unit $\eta$ of the adjunction $\left(-\otimes_{\Lambda}^{\mathbb L}\Gamma, \res\right)$ is an isomorphism on the level of singularity categories. For this purpose, notice that the triangle
\[
\Lambda \xrightarrow{f}\Gamma \to \Cone(f) \to \Lambda[1]
\]
in $\Db(\modules \Lambdae)$ will induce, for each $X\in\Db(\modules \Lambda)$, the triangle
\[
X\otimes_{\Lambda}^{\mathbb L}\Cone(f)[-1] \to X \xrightarrow{\eta_X} X\otimes_{\Lambda}^{\mathbb L}\Gamma \to X\otimes_{\Lambda}^{\mathbb L}\Cone(f)
\]
in $\Db(\modules \Lambda)$. By assumption the outer terms belong to $\perf \Lambda$, which means that $\eta$ becomes an isomorphism in the quotient category $\Dsg(\Lambda)$ as desired. Claim (ii) is now immediate, since the stable categories of Gorenstein projective modules are full subcategories of the respective singularity categories.
\end{proof}

We collect an immediate consequence, using Proposition~\ref{prop:coneperfectimpliesfullyfaithful} in order to view $\Gstable\Lambda$ and $\Dsg(\Lambda)$ as subcategories of $\Gstable\Gamma$ and $\Dsg(\Gamma)$, respectively.

\begin{corollary} \label{cor:twotorsionpair}
Suppose $\Cone(f)\in\perf\Lambdae$. Then we have the following.
\begin{enumerate}[label=\emph{(\roman*)}]
\item The pair of subcategories
\[
\left(\Dsg(\Lambda), \Ker (\res_{\Dsg}) \right)
\]
is a stable $t$-structure in $\Dsg(\Gamma)$. In particular, the fully faithful functor $-\otimes_{\Lambda}^{\mathbb L}\Gamma$ induces a triangle equivalence
\[
\Dsg(\Lambda)\cong\Dsg(\Gamma)/\Ker(\res_{\Dsg}).
\]
\item If moreover $\RHom_{\Lambda}(\Gamma, \Lambda)$ lies in $\perf \Gamma$, then the pair of subcategories
\[
\left(\Gstable \Lambda, \Ker (\res) \right)
\]
is a stable $t$-structure in $\Gstable \Gamma $. In this case the fully faithful functor $-\otimes_{\Lambda}\Gamma$ induces a triangle equivalence
\[
\Gstable\Lambda\cong\Gstable\Gamma/\Ker(\res).
\]
\end{enumerate}
\end{corollary}

Let us include an easy application. Recall from \cite{MR1140607} that $f$ is a \emph{homological epimorphism} if $\res\colon\Db(\modules\Gamma) \to \Db(\modules \Lambda)$ is fully faithful. In this case also the functor $\res\colon\Dsg(\modules\Gamma) \to \Dsg(\modules \Lambda)$ is fully faithful by Lemma~\ref{lemadjointquotient}. In particular, the kernel of the latter then vanishes, so Corollary~\ref{cor:twotorsionpair} reveals the following.

\begin{corollary}
\label{chen'sresult}
Suppose $f$ is a homological epimorphism and that $\Cone(f)\in\perf\Lambdae$. Then $\Lambda$ and $\Gamma$ are singularly equivalent.
\end{corollary}

Observe that the main result of \cite{MR3209319} is recovered as the case $f\colon\Lambda \to \Lambda/I$ for $\Lambda$ a finite dimensional algebra with a homological ideal $I$ of finite projective dimension as bimodule.

\subsection*{Examples}

We illustrate Proposition~\ref{prop:coneperfectimpliesfullyfaithful} and Corollary~\ref{cor:twotorsionpair} with two examples --- both of which will first be presented in a generic form and then illustrated completely explicitly for certain Nakayama algebras.

\begin{corollary} \label{cor.ideal is tensor}
Let $\Lambda$ be a module finite algebra over some regular commutative noetherian ring $k$. Suppose $I$ is an ideal of $\Lambda$ which, as a bimodule, is isomorphic to a tensor product \[I \cong M \otimes_k N,\] where $M$ and $N$ are left and right $\Lambda$-modules, respectively, both of finite projective dimension. Then the functors
\[
-\otimes_{\Lambda}^{\mathbb L} \Lambda/I \colon\Dsg(\Lambda)\to\Dsg(\Lambda / I) \text{ and }-\otimes_{\Lambda}\Lambda/I\colon\Gstable \Lambda \to \Gstable \Lambda / I
\]
are both fully faithful.
\end{corollary}

\begin{proof}
The cone of $\Lambda \to \Lambda/I$ is $I[1]$. Thus, by Proposition~\ref{prop:coneperfectimpliesfullyfaithful} it suffices for $I$ to be a perfect bimodule. This is guaranteed if $I$ is a tensor product of two perfect modules.
\end{proof}

\begin{corollary}\label{cor.1-dim ideal}
Let $\Lambda$ be a finite dimensional algebra over field $k$, and let $I$ be a $1$-dimensional ideal. Assume that the projective dimension of $I$ is finite both as left and as right module. Then the functors
\[
-\otimes_{\Lambda}^{\mathbb L} \Lambda/I \colon\Dsg(\Lambda)\to\Dsg(\Lambda / I)\text{ and }-\otimes_{\Lambda}\Lambda/I\colon\Gstable \Lambda \to \Gstable \Lambda / I
\] are both fully faithful.
\end{corollary}

\begin{proof}
Since $I$ is $1$-dimensional, we have $I \cong I \otimes_k I$ as bimodules. Thus the claim follows from Corollary~\ref{cor.ideal is tensor} above.
\end{proof}

\begin{example}
\label{firstexample}
Let $k$ be a field and $\Lambda = k[\!\!\!
\begin{tikzpicture}[baseline=-1mm]
 \matrix (m) [matrix of math nodes, column sep = .6cm] {
  1 & 2 \\ };
  \draw [->, bend right] (m-1-1) to node [below] {$\scriptstyle a$} (m-1-2);
  \draw [->, bend right] (m-1-2) to node [above] {$\scriptstyle b$} (m-1-1);
\end{tikzpicture}
\!\!\!] / (ab)^n$ for some $n > 0$.

Consider the ideal $I = (ba)^n$. Then $I$ is isomorphic to the simple at vertex $1$ both as left and as right module, and moreover these simples have projective dimension $1$. Thus Corollary~\ref{cor.1-dim ideal} applies, and we obtain fully faithful functors
\[
-\otimes_{\Lambda}^{\mathbb L} \Lambda/I \colon\Dsg(\Lambda)\to\Dsg(\Lambda / I) \text{ and }
-\otimes_{\Lambda}\Lambda/I \colon\Gstable \Lambda \to \Gstable \Lambda / I.
\]
In fact, in this example both $\Lambda$ and $\Lambda / I$ are Gorenstein, so the singularity categories coincide with the respective stable categories of Gorenstein projective modules. Moreover, $\Lambda / I$ is even self-injective, so $\Dsg( \Lambda / I) = \underline{\modules}\,\Lambda / I$. One may observe that
\[
\Ker( \res_{\Dsg} ) = \{ M \in \underline{\modules}(\Lambda / I) \mid \pdim M_{\Lambda}<\infty \} = \add \{ S_1, P_2 / S_1 \}.
\]
It now follows from Corollary~\ref{cor:twotorsionpair} that
\[
\Dsg(\Lambda)  = \prescript{\perp}{}{\Ker( \res_{\Dsg} )} = \add \{ P_2 / \rad^2 P_2, P_2 / \rad^4 P_2, \ldots, P_2 / \rad^{2n-2} P_2 \}.
\]
We illustrate the subcategories inside the Auslander--Reiten quiver for $n=3$:
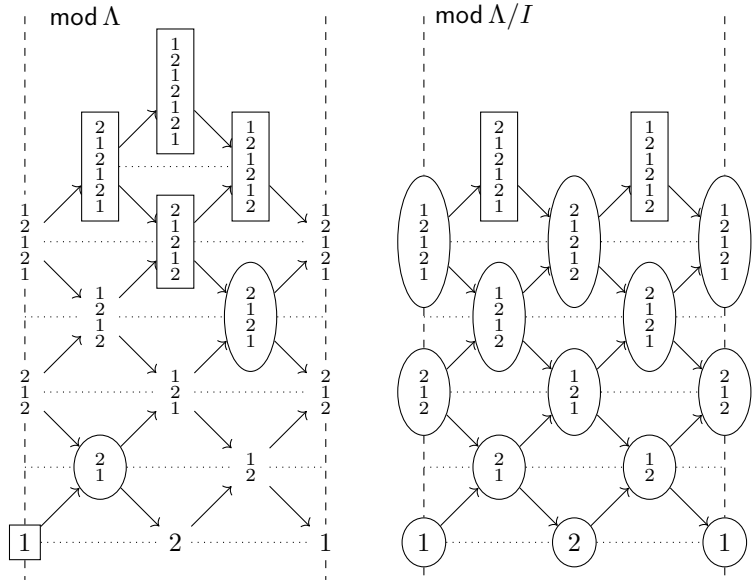
\[ \begin{tikzpicture}
 \node at (.8, 7) {$\modules \Lambda$};
 \node [draw] (1) at (0,0) {$1$};
 \node (2) at (2,0) {$2$};
 \node (1+) at (4,0) {$1$};
 \node (12) at (3,1) {$\begin{smallmatrix} 1 \\ 2 \end{smallmatrix}$};
 \node (21) [draw, ellipse] at (1,1) {$\begin{smallmatrix} 2 \\ 1 \end{smallmatrix}$};
 \node (121) at (2,2) {$\begin{smallmatrix} 1 \\ 2 \\ 1 \end{smallmatrix}$};
 \node (212) at (0,2) {$\begin{smallmatrix} 2 \\ 1 \\ 2 \end{smallmatrix}$};
 \node (212+) at (4,2) {$\begin{smallmatrix} 2 \\ 1 \\ 2 \end{smallmatrix}$};
 \node (1212) at (1,3) {$\begin{smallmatrix} 1 \\ 2 \\ 1 \\ 2 \end{smallmatrix}$};
 \node (2121) [draw, ellipse] at (3,3) {$\begin{smallmatrix} 2 \\ 1 \\ 2 \\ 1 \end{smallmatrix}$};
 \node (12121) at (0,4) {$\begin{smallmatrix} 1 \\ 2 \\ 1 \\ 2 \\ 1 \end{smallmatrix}$};
 \node (21212) [draw] at (2,4) {$\begin{smallmatrix} 2 \\ 1 \\ 2 \\ 1 \\ 2 \end{smallmatrix}$};
 \node (12121+) at (4,4) {$\begin{smallmatrix} 1 \\ 2 \\ 1 \\ 2 \\ 1 \end{smallmatrix}$};
 \node (121212) [draw] at (3,5) {$\begin{smallmatrix} 1 \\ 2 \\ 1 \\ 2 \\ 1 \\ 2 \end{smallmatrix}$};
 \node (212121) [draw] at (1,5) {$\begin{smallmatrix} 2 \\ 1 \\ 2 \\ 1 \\ 2 \\ 1 \end{smallmatrix}$};
 \node (1212121) [draw] at (2,6) {$\begin{smallmatrix} 1 \\ 2 \\ 1 \\ 2 \\ 1 \\ 2 \\ 1 \end{smallmatrix}$};
 \draw [->] (1) to (21);
 \draw [->] (21) to (2);
 \draw [->] (2) to (12);
 \draw [->] (12) to (1+);
 \draw [->] (212) to (21);
 \draw [->] (21) to (121);
 \draw [->] (121) to (12);
 \draw [->] (12) to (212+);
 \draw [->] (212) to (1212);
 \draw [->] (1212) to (121);
 \draw [->] (121) to (2121);
 \draw [->] (2121) to (212+);
 \draw [->] (12121) to (1212);
 \draw [->] (1212) to (21212);
 \draw [->] (21212) to (2121);
 \draw [->] (2121) to (12121+);
 \draw [->] (12121) to (212121);
 \draw [->] (212121) to (21212);
 \draw [->] (21212) to (121212);
 \draw [->] (121212) to (12121+);
 \draw [->] (212121) to (1212121);
 \draw [->] (1212121) to (121212);
 \draw [dotted] (1) to (2);
 \draw [dotted] (2) to (1+);
 \draw [dotted] (0,1) to (21);
 \draw [dotted] (21) to (12);
 \draw [dotted] (12) to (4,1);
 \draw [dotted] (212) to (121);
 \draw [dotted] (121) to (212+);
 \draw [dotted] (0,3) to (1212);
 \draw [dotted] (1212) to (2121);
 \draw [dotted] (2121) to (4,3);
 \draw [dotted] (12121) to (21212);
 \draw [dotted] (21212) to (12121+);
 \draw [dotted] (212121) to (121212);
 \draw [dashed] (0,7) -- (12121);
 \draw [dashed] (12121) -- (212);
 \draw [dashed] (212) -- (1);
 \draw [dashed] (1) -- (0,-.5);
 \draw [dashed] (4,7) -- (12121+);
 \draw [dashed] (12121+) -- (212+);
 \draw [dashed] (212+) -- (1+);
 \draw [dashed] (1+) -- (4,-.5);
\end{tikzpicture}
\qquad
\begin{tikzpicture}
 \node at (.8, 7) {$\modules \Lambda/I$};
 \node (1) [draw, ellipse] at (0,0) {$1$};
 \node (2) [draw, ellipse] at (2,0) {$2$};
 \node (1+) [draw, ellipse] at (4,0) {$1$};
 \node (12) [draw, ellipse] at (3,1) {$\begin{smallmatrix} 1 \\ 2 \end{smallmatrix}$};
 \node (21) [draw, ellipse] at (1,1) {$\begin{smallmatrix} 2 \\ 1 \end{smallmatrix}$};
 \node (121) [draw, ellipse] at (2,2) {$\begin{smallmatrix} 1 \\ 2 \\ 1 \end{smallmatrix}$};
 \node (212) [draw, ellipse] at (0,2) {$\begin{smallmatrix} 2 \\ 1 \\ 2 \end{smallmatrix}$};
 \node (212+) [draw, ellipse] at (4,2) {$\begin{smallmatrix} 2 \\ 1 \\ 2 \end{smallmatrix}$};
 \node (1212) [draw, ellipse] at (1,3) {$\begin{smallmatrix} 1 \\ 2 \\ 1 \\ 2 \end{smallmatrix}$};
 \node (2121) [draw, ellipse] at (3,3) {$\begin{smallmatrix} 2 \\ 1 \\ 2 \\ 1 \end{smallmatrix}$};
 \node (12121) [draw, ellipse] at (0,4) {$\begin{smallmatrix} 1 \\ 2 \\ 1 \\ 2 \\ 1 \end{smallmatrix}$};
 \node (21212) [draw, ellipse] at (2,4) {$\begin{smallmatrix} 2 \\ 1 \\ 2 \\ 1 \\ 2 \end{smallmatrix}$};
 \node (12121+) [draw, ellipse] at (4,4) {$\begin{smallmatrix} 1 \\ 2 \\ 1 \\ 2 \\ 1 \end{smallmatrix}$};
 \node (121212) [draw] at (3,5) {$\begin{smallmatrix} 1 \\ 2 \\ 1 \\ 2 \\ 1 \\ 2 \end{smallmatrix}$};
 \node (212121) [draw] at (1,5) {$\begin{smallmatrix} 2 \\ 1 \\ 2 \\ 1 \\ 2 \\ 1 \end{smallmatrix}$};
 \draw [->] (1) to (21);
 \draw [->] (21) to (2);
 \draw [->] (2) to (12);
 \draw [->] (12) to (1+);
 \draw [->] (212) to (21);
 \draw [->] (21) to (121);
 \draw [->] (121) to (12);
 \draw [->] (12) to (212+);
 \draw [->] (212) to (1212);
 \draw [->] (1212) to (121);
 \draw [->] (121) to (2121);
 \draw [->] (2121) to (212+);
 \draw [->] (12121) to (1212);
 \draw [->] (1212) to (21212);
 \draw [->] (21212) to (2121);
 \draw [->] (2121) to (12121+);
 \draw [->] (12121) to (212121);
 \draw [->] (212121) to (21212);
 \draw [->] (21212) to (121212);
 \draw [->] (121212) to (12121+);
 \draw [dotted] (1) to (2);
 \draw [dotted] (2) to (1+);
 \draw [dotted] (0,1) to (21);
 \draw [dotted] (21) to (12);
 \draw [dotted] (12) to (4,1);
 \draw [dotted] (212) to (121);
 \draw [dotted] (121) to (212+);
 \draw [dotted] (0,3) to (1212);
 \draw [dotted] (1212) to (2121);
 \draw [dotted] (2121) to (4,3);
 \draw [dotted] (12121) to (21212);
 \draw [dotted] (21212) to (12121+);
 \draw [dashed] (0,7) -- (12121);
 \draw [dashed] (12121) -- (212);
 \draw [dashed] (212) -- (1);
 \draw [dashed] (1) -- (0,-.5);
 \draw [dashed] (4,7) -- (12121+);
 \draw [dashed] (12121+) -- (212+);
 \draw [dashed] (212+) -- (1+);
 \draw [dashed] (1+) -- (4,-.5);
\end{tikzpicture} \]
In both cases, the dashed lines are identified. The rectangles mark modules of finite projective dimension (including the projective modules at the top), and the ellipses mark non-projective Gorenstein projective modules. In the picture clearly the ellipses on the left are a subset of those on the right.
\end{example}

\begin{corollary} \label{cor.trivialext}
Let $\Lambda$ be module finite over a regular commutative noetherian ring $k$. Assume $M$ and $N$ are left and right $\Lambda$-modules of finite projective dimension, respectively. Set
\[ \Gamma = \Lambda \ltimes (M \otimes_k N). \]
Then there are fully faithful functors
\[
-\otimes_{\Lambda}^{\mathbb L}\Gamma\colon\Dsg(\Lambda)\to\Dsg(\Gamma) \text{ and } -\otimes_{\Lambda}\Gamma\colon\Gstable \Lambda \to \Gstable \Gamma.
\]
\end{corollary}

\begin{proof}
This is an immediate application of Proposition~\ref{prop:coneperfectimpliesfullyfaithful} above.
\end{proof}

\begin{example}
\label{secondexample}
Let $k$ be a field and $\Lambda = k\Bigg[\!\!\!
\begin{tikzpicture}[baseline=-0mm]
 \matrix (m) [matrix of math nodes, column sep = .2cm] {
  1 && 2 \\ & 3 & \\};
  \draw [->] (m-1-1) to node [above] {$\scriptstyle a$} (m-1-3);
  \draw [->] (m-1-3) to node [right, shift={(-.05,-.2)}] {$\scriptstyle b$} (m-2-2);
  \draw [->] (m-2-2) to node [left, shift={(.1,-.2)}] {$\scriptstyle c$} (m-1-1);
\end{tikzpicture}\!\!\!\Bigg] \big/ (cb, bac)$.

We see that the simple at vertex $2$ has projective dimension $1$ as right $\Lambda$-module, and so does the simple at vertex $1$ as left $\Lambda$-module. Therefore, for
\[
\Gamma = \Lambda \ltimes (S_1^{\mathsf{left}} \otimes_k S_2) = k\Bigg[\!\!\!
\begin{tikzpicture}[baseline=-0mm]
 \matrix (m) [matrix of math nodes, column sep = .2cm] {
  1 && 2 \\ & 3 & \\};
  \draw [->] (m-1-1) to node [above] {$\scriptstyle a$} (m-1-3);
  \draw [->] (m-1-3) to node [right, shift={(-.05,-.2)}] {$\scriptstyle b$} (m-2-2);
  \draw [->] (m-2-2) to node [left, shift={(.1,-.2)}] {$\scriptstyle c$} (m-1-1);
  \draw [->,bend right=60] (m-1-3) to node [above] {$\scriptstyle x$} (m-1-1);
\end{tikzpicture}\!\!\!\Bigg] \big/ (cb, bac,ax, xa),
\]
Corollary~\ref{cor.trivialext} applies and we obtain fully faithful functors
\[
-\otimes_{\Lambda}^{\mathbb L} \Gamma  \colon\Dsg(\Lambda)\to\Dsg(\Gamma) \text{ and }
-\otimes_{\Lambda}\Gamma  \colon\Gstable \Lambda \to \Gstable \Gamma.
\]
As before, we illustrate the relevant categories inside the Auslander--Reiten quivers in order to get an immediate visual confirmation of the result.
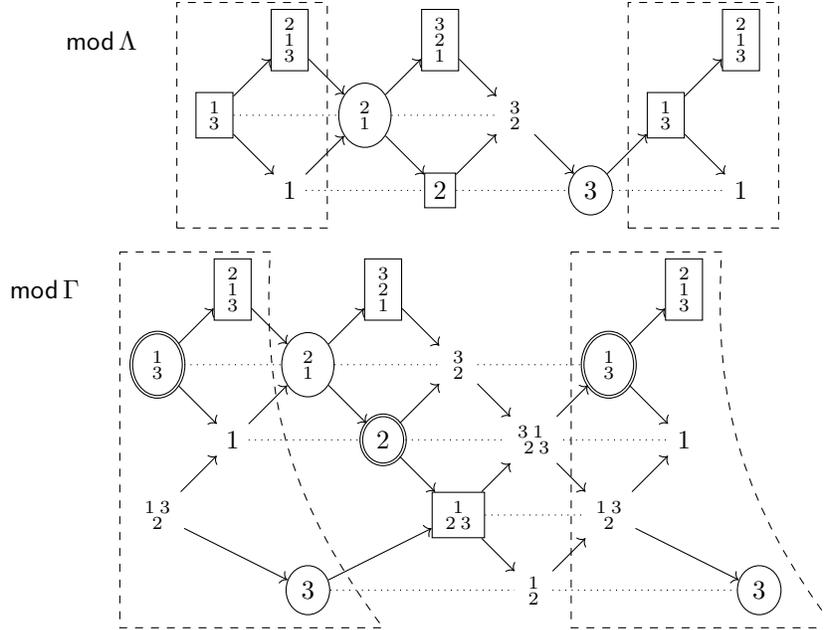
\[ \begin{tikzpicture}
 \node at (-2.5, 2) {$\modules \Lambda$};
 \node (1) at (0,0) {$1$};
 \node (2) [draw] at (2,0) {$2$};
 \node (3) [draw, ellipse] at (4,0) {$3$};
 \node (1+) at (6,0) {$1$};
 \node (13-) [draw] at (-1,1) {$\begin{smallmatrix} 1 \\ 3 \end{smallmatrix}$};
 \node (32) at (3,1) {$\begin{smallmatrix} 3 \\ 2 \end{smallmatrix}$};
 \node (21) [draw, ellipse] at (1,1) {$\begin{smallmatrix} 2 \\ 1 \end{smallmatrix}$};
 \node (13) [draw] at (5,1) {$\begin{smallmatrix} 1 \\ 3 \end{smallmatrix}$};
 \node (321) [draw] at (2,2) {$\begin{smallmatrix} 3 \\ 2 \\ 1 \end{smallmatrix}$};
 \node (213) [draw] at (0,2) {$\begin{smallmatrix} 2 \\ 1 \\ 3 \end{smallmatrix}$};
 \node (213+) [draw] at (6,2) {$\begin{smallmatrix} 2 \\ 1 \\ 3 \end{smallmatrix}$};
 \draw [->] (13-) to (1);
 \draw [->] (1) to (21);
 \draw [->] (21) to (2);
 \draw [->] (2) to (32);
 \draw [->] (32) to (3);
 \draw [->] (3) to (13);
 \draw [->] (13) to (1+);
 \draw [->] (13-) to (213);
 \draw [->] (213) to (21);
 \draw [->] (21) to (321);
 \draw [->] (321) to (32);
 \draw [->] (13) to (213+);
 \draw [dotted] (1) to (2);
 \draw [dotted] (2) to (3);
 \draw [dotted] (3) to (1+);
 \draw [dotted] (13-) to (21);
 \draw [dotted] (21) to (32);
% \draw [dotted] (32) to (13);
% \draw [dotted] (213) to (321);
 \draw [dashed] (-1.5, 2.5) rectangle (0.5, -.5);
 \draw [dashed] (4.5, 2.5) rectangle (6.5, -.5);
\end{tikzpicture} \]
\[ \begin{tikzpicture}
 \node at (-2.5, 2) {$\modules \Gamma$};
 \node (1) at (0,0) {$1$};
 \node (2) [draw, ellipse, double] at (2,0) {$2$};
 \node (3213) at (4,0) {$\begin{smallmatrix} 3 \: 1 \; \\ \; 2 \: 3 \end{smallmatrix}$};
 \node (1+) at (6,0) {$1$};
 \node (13-) [draw, ellipse, double] at (-1,1) {$\begin{smallmatrix} 1 \\ 3 \end{smallmatrix}$};
 \node (32) at (3,1) {$\begin{smallmatrix} 3 \\ 2 \end{smallmatrix}$};
 \node (21) [draw, ellipse] at (1,1) {$\begin{smallmatrix} 2 \\ 1 \end{smallmatrix}$};
 \node (13) [draw, ellipse, double] at (5,1) {$\begin{smallmatrix} 1 \\ 3 \end{smallmatrix}$};
 \node (321) [draw] at (2,2) {$\begin{smallmatrix} 3 \\ 2 \\ 1 \end{smallmatrix}$};
 \node (213) [draw] at (0,2) {$\begin{smallmatrix} 2 \\ 1 \\ 3 \end{smallmatrix}$};
 \node (213+) [draw] at (6,2) {$\begin{smallmatrix} 2 \\ 1 \\ 3 \end{smallmatrix}$};
 \node (132-) at (-1,-1) {$\begin{smallmatrix} 1 \: 3 \\ 2 \end{smallmatrix}$};
 \node (123) [draw] at (3,-1) {$\begin{smallmatrix} 1 \\ 2 \: 3 \end{smallmatrix}$};
 \node (132) at (5,-1) {$\begin{smallmatrix} 1 \: 3 \\ 2 \end{smallmatrix}$};
 \node (3) [draw, ellipse] at (1,-2) {$3$};
 \node (12) at (4,-2) {$\begin{smallmatrix} 1 \\ 2 \end{smallmatrix}$};
 \node (3+) [draw, ellipse] at (7, -2) {$3$};
 \draw [->] (13-) to (1);
 \draw [->] (1) to (21);
 \draw [->] (21) to (2);
 \draw [->] (2) to (32);
 \draw [->] (32) to (3213);
 \draw [->] (3213) to (13);
 \draw [->] (13) to (1+);
 \draw [->] (13-) to (213);
 \draw [->] (213) to (21);
 \draw [->] (21) to (321);
 \draw [->] (321) to (32);
 \draw [->] (13) to (213+);
 \draw [->] (132-) to (1);
 \draw [->] (2) to (123);
 \draw [->] (123) to (3213);
 \draw [->] (3213) to (132);
 \draw [->] (132) to (1+);
 \draw [->] (132-) to (3);
 \draw [->] (3) to (123);
 \draw [->] (123) to (12);
 \draw [->] (12) to (132);
 \draw [->] (132) to (3+);
 \draw [dotted] (1) to (2);
 \draw [dotted] (2) to (3213);
 \draw [dotted] (3213) to (1+);
 \draw [dotted] (13-) to (21);
 \draw [dotted] (21) to (32);
 \draw [dotted] (32) to (13);
% \draw [dotted] (213) to (321);
 \draw [dotted] (123) to (132);
 \draw [dotted] (3) to (12);
 \draw [dotted] (12) to (3+);
 \draw [dashed] (-1.5, 2.5) to (0.5,2.5) [bend right = 20] to (2, -2.5) [bend right=0] to (-1.5,-2.5) to cycle;
 \draw [dashed] (4.5, 2.5) to (6.5,2.5) [bend right = 20] to (8, -2.5) [bend right=0] to (4.5,-2.5) to cycle;
\end{tikzpicture} \]
The dashed boxes indicate which parts of the quivers are identified. The rectangles indicate modules of finite projective dimension --- for $\Gamma$ these are only the projective modules. In particular, one sees that not all injective modules are inside rectangles, that is the algebras are not Gorenstein. The modules inside ellipses are representatives of the indecomposable objects in the singularity category. (For both algebras, all projective resolutions are eventually periodic, so the singularity categories are Krull--Schmidt with indecomposable objects being represented by the infinite syzygies. However, the singularity category of $\Gamma$ is not a subcategory of the stable module category: Contrary to what the picture indicates, there is no non-zero morphism from $\begin{smallmatrix} 2 \\ 1 \end{smallmatrix}$ to $2$ in the singularity category of $\Gamma$, and this singularity category is in fact semisimple.) The doubly circled modules are the non-projective Gorenstein projective modules. One immediately verifies that the circles in the upper picture are a subcollection of the circles in the lower one, and similarly --- albeit vacuously --- for double circles.
\end{example}

\section{Big singularity categories} \label{sec:bigcat}

We now investigate to what extent change of rings gives rise to functors between the corresponding homotopy categories of projective or injective modules, and prominent subcategories thereof. In particular, the results of this section show that the diagram in Theorem~\ref{thm finite main} is in fact the restriction of a picture that even exists on the level of the compactly generated completions from Subsection~\ref{subsection:compactcompletions}.

The functors $\rho$ and $\lambda$ will occur frequently, and we write $\widetilde\Omega = \Omega \circ[1]$ and $\widetilde\mho = \mho\circ[-1]$ for the sake of brevity --- confer Proposition~\ref{prop nice proj resol}.

\subsection*{Full homotopy categories}

We start by demonstrating how from a morphism of rings subject only to rather lenient assumptions, an abundance of adjoints are induced on the level of full homotopy categories.
%$d_{\text{left}}$ and $d_{\text{right}}$

\begin{proposition} \label{prop adj for K(Inj)}
Let $f\colon \Lambda \to \Gamma$ be a morphism of rings. If $d = \pdim {}_{\Lambda}\Gamma < \infty$, then there is an adjoint triple of functors
\[ \begin{tikzpicture}
 \node (KIA) at (-1,1) {$\K(\Inj \Lambda)$};
 \node (KIB) at (4,1) {$\K(\Inj \Gamma).$};
 \draw [->,bend left=60] (KIA) to node [above] {$\scriptstyle \lambda (\widetilde\Omega^d - \otimes_{\Lambda} \Gamma)$} (KIB);
 \draw [->,bend right] (KIB) to node [above] {$\scriptstyle \lambda \res$} (KIA);
 \draw [->] (KIA) to node [above] {$\scriptstyle \Hom_{\Lambda}(\Gamma, -)$} (KIB);
\end{tikzpicture} \]

Moreover, if $d' = \pdim \Gamma_{\Lambda} < \infty$, then there is an adjoint triple of functors
\[ \begin{tikzpicture}
 \node (KPA) at (0,0) {$\K(\Proj \Lambda)$};
 \node (KPB) at (5,0) {$\K(\Proj \Gamma).$};
 \draw [->] (KPA) to node [above] {$\scriptstyle - \otimes_{\Lambda} \Gamma$} (KPB);
 \draw [->,bend left] (KPB) to node [above] {$\scriptstyle \rho \res$} (KPA);
 \draw [->,bend right=60] (KPA) to node [above=.5mm] {$\scriptstyle \rho \Hom_{\Lambda}(\Gamma, \widetilde\mho^{d'} - )$} (KPB);
\end{tikzpicture} \]
\end{proposition}

\begin{proof}
We only discuss the first claim. The second one is dual, and in fact easier since it avoids the slight extra problem of translating left projective to right injective dimension.

First observe that $\Hom_{\Lambda}(\Gamma, -)$ maps injective $\Lambda$-modules to injective $\Gamma$-modules. Indeed $\Hom_{\Gamma}(-, \Hom_{\Lambda}(\Gamma, I)) \cong \Hom_{\Lambda}(\res -, I)$ is exact for any injective $\Lambda$-module $I$. Therefore the bottom functor is well-defined.

Since $\res$ is left adjoint to $\Hom_{\Lambda}( \Gamma, - )$ on the level of homotopy categories of all modules, and $\lambda$ is left adjoint to the inclusion of the homotopy category of injectives into that category, $\lambda \res$ is left adjoint to $\Hom_{\Lambda}(\Gamma, -)$ on the level of homotopy categories of injectives.

Next we observe that the right injective dimension of $\res I$ for injective $\Gamma$-modules $I$ is bounded by $d$. Indeed, since $-\otimes_\Lambda \Gamma$ is left adjoint to $\res$ we have an isomorphism $\Hom_{\Lambda}(-, \res I) \cong \Hom_{\Gamma}(- \otimes_{\Lambda} \Gamma, I)$, and the latter has at most $d$ derived functors.

We can thus apply Propositions~\ref{prop nice proj resol} and \ref{prop syzygy-cosyzygy} in the first two lines below, and obtain the following isomorphisms$\colon$
\begin{align*}
\Hom_{\K(\Inj \Lambda)}( X, \lambda \res Y) & = \Hom_{\K(\Inj \Lambda)}( X, \widetilde\mho^d \res Y) \\
& =  \Hom_{\K(\Lambda)}( \widetilde\Omega^d X , \res Y) \\
& = \Hom_{\K(\Gamma)}( \widetilde\Omega^d X \otimes_{\Lambda} \Gamma , Y) \\
& = \Hom_{\K(\Inj \Gamma)}( \lambda ( \widetilde\Omega^d X \otimes_{\Lambda} \Gamma ), Y). %\qedhere
\end{align*}
We infer that $\lambda ( \widetilde\Omega^d - \otimes_{\Lambda} \Gamma )$ is left adjoint to $\lambda \res$.
\end{proof}

\begin{proposition}
\label{propfivetupleadj}
Let $f\colon \Lambda \to \Gamma$ be a morphism of noetherian algebras with $d = \pdim {}_{\Lambda}\Gamma < \infty$. Then there is an adjoint quadruple as indicated by the solid arrows in the following diagram. If also $\pdim \Gamma_{\Lambda} < \infty $, then we even have a five-tuple of adjoint functors.
\[ \begin{tikzpicture}
 \node (KIA) at (-1,1) {$\K(\Inj \Lambda)$};
 \node (KIB) at (4,1) {$\K(\Inj \Gamma)$};
 \draw [->,bend left=60] (KIA) to node [above] {$\scriptstyle \lambda (\widetilde\Omega^d - \otimes_{\Lambda} \Gamma)$} (KIB);
 \draw [->,bend right] (KIB) to node [above] {$\scriptstyle \lambda \res$} (KIA);
 \draw [->] (KIA) to node [above] {$\scriptstyle \Hom_{\Lambda}(\Gamma, -)$} (KIB);
 \draw [->,bend left] (KIB) to (KIA);
 \draw [->,bend right=60,dashed] (KIA) to (KIB);
\end{tikzpicture} \]
\end{proposition}

\begin{proof}
Consider first the functor $\Hom_{\Lambda}(\Gamma, -)$. We observe that its left adjoint $\lambda \res$ restricts to the respective subcategories of compacts. Therefore $\Hom_{\Lambda}(\Gamma, -)$ has a right adjoint by Theorem~\ref{thm:neemanthreeadjoints}.

Similarly, for the bottom solid functor, its left adjoint $\Hom_{\Lambda}(\Gamma, -)$ preserves compact objects if and only if $\Gamma_{\Lambda}$ has finite projective dimension. Thus, in this case we obtain the additional dashed adjoint.
\end{proof}

Unfortunately, the two new functors in the above proposition are not explicit. However, in the case that our rings are module finite over some commutative Gorenstein ring, we get a much more explicit picture.

\begin{proposition} \label{prop:commutativeoverkwithadjoints}
Let $f\colon \Lambda \to \Gamma$ be a morphism of module finite algebras over a commutative noetherian Gorenstein ring. Assume that $d = \pdim {}_{\Lambda}\Gamma < \infty$ and that $d' = \pdim \Gamma_{\Lambda} < \infty$.
%$d = \maxx \{ \pdim \:_{\Lambda} \Gamma, \pdim \Gamma_{\Lambda} \} < \infty$.
Then we have the following diagram of adjoints, with a commutative square in the middle.
\[ \begin{tikzpicture}
 \node (KIA) at (0,1.5) {$\K(\Inj \Lambda)$};
 \node (KIB) at (5,1.5) {$\K(\Inj \Gamma)$};
 \node (KPA) at (0,0) {$\K(\Proj \Lambda)$};
 \node (KPB) at (5,0) {$\K(\Proj \Gamma)$};
 \draw [<-] (KIA) to node [left] {$\approx$} node [right] {$\scriptstyle - \otimes_{\Lambda} D_\Lambda$} (KPA);
 \draw [<-] (KIB) to node [left] {$\approx$} node [right] {$\scriptstyle - \otimes_{\Gamma} D_\Gamma$} (KPB);
 \draw [->,bend left=60] (KIA) to node [above] {$\scriptstyle \lambda(\widetilde\Omega^d - \otimes_{\Lambda} \Gamma)$} (KIB);
 \draw [->,bend right] (KIB) to node [above] {$\scriptstyle \lambda \res$} (KIA);
 \draw [->] (KIA) to node [above] {$\scriptstyle \Hom_{\Lambda}(\Gamma, -)$} (KIB);
 \draw [->] (KPA) to node [above] {$\scriptstyle - \otimes_{\Lambda} \Gamma$} (KPB);
 \draw [->,bend left] (KPB) to node [above] {$\scriptstyle \rho \res$} (KPA);
 \draw [->,bend right=60] (KPA) to node [above=.5mm] {$\scriptstyle \rho \Hom_{\Lambda}(\Gamma, \widetilde\mho^{d'} - )$} (KPB);
\end{tikzpicture} \]
\end{proposition}

\begin{proof}
This is a combination of Proposition~\ref{prop adj for K(Inj)} (for the adjunctions) and Proposition~\ref{prop duality commutes} (for the commutative square).
\end{proof}

\subsection*{Acyclic and coacyclic complexes}

We now consider the subcategories of acyclic and coacyclic complexes. Our strategy is to investigate which of the above established functors restrict to this level under which conditions.

\begin{proposition} \label{prop restriction to acyclic}
Let $f\colon \Lambda \to \Gamma$ be a morphism of rings.
Assume that $d = \pdim \Gamma_{\Lambda}<\infty$.
Then we have the following.
\begin{enumerate}[label=\emph{(\roman*)}]
\item The functors $\Hom_{\Lambda}( \Gamma, - )$ and $\lambda \res$ restrict to functors between homotopy categories of acyclic complexes of injectives.
\item The functor $- \otimes_{\Lambda} \Gamma$ restricts to a functor between homotopy categories of coacyclic complexes of projectives.
\item If additionally for any projective $\Lambda$-module $Q$ the complex $\RHom_{\Lambda}(\Gamma, Q)$ is quasi-isomorphic to a bounded complex of projective $\Gamma$-modules, then the functor $\rho \res$ restricts to a functor between categories of coacyclic complexes of projectives.

If $\Gamma$ is perfect as $\Lambda$-module, the condition is satisfied when $\RHom_{\Lambda}(\Gamma, \Lambda)$ lies in $\perf \Gamma$.
\end{enumerate}
\end{proposition}

\begin{proof}
For the first part of (i) observe that if $I \in \K_{\ac}( \Inj \Lambda)$, then
\[
\HH^i( \Hom_{\Lambda}(\Gamma, I)) = \Ext_{\Lambda}^{i+j}(\Gamma, \B^{-j}(I)) = 0 \text{ for } j \gg 0
\]
since $\pd \Gamma_{\Lambda}<\infty$. For the second part of (i) note that $\res$ preserves acyclicity, and so does $\lambda$ by the construction in Proposition~\ref{prop nice proj resol}.

For (ii), let $P \in \K_{\coac}(\Proj \Lambda)$. For $Q \in \Proj \Gamma$ we have
\[ \Hom_{\K(\Gamma)}(P \otimes_{\Lambda} \Gamma, Q) = \Hom_{\K(\Lambda)}(P, Q), \]
which vanishes since $\Hom_{\K(\Lambda)}(P, \Proj\Lambda[i])=0$, and $Q$ has a finite projective resolution over $\Lambda$.

For (iii), we let $P\in\K_{\coac}(\Proj \Gamma)$ and look at $\Hom_{\K(\Lambda)}( \rho \res P, Q)$. By adjunction on the level of full homotopy categories of projectives we have
\[ \Hom_{\K(\Lambda)}( \rho \res P, Q) = \Hom_{\K(\Gamma)}(P, \rho \Hom_{\Lambda}(\Gamma, \widetilde\mho^d Q)). \]
Note that the second argument in the $\Hom_{\K(\Gamma)}$ is just a projective resolution of $\RHom_{\Lambda}(\Gamma, Q)$. As this is finite by assumption, the $\Hom$-space vanishes since $\Hom_{\K(\Gamma)}(P, -)$ vanishes on all projectives.

Finally, if $\Gamma$ is perfect as a $\Lambda$-module then $\RHom_{\Lambda}\left(\Gamma, \Lambda^{(I)}\right)\cong \RHom_\Lambda\left(\Gamma, \Lambda\right)^{(I)}$, whence it suffices to consider the projective $\Lambda$-module $\Lambda$.
\end{proof}

Observe that there is a certain asymmetry in the above, in the sense that $\lambda \res$ automatically restricts, while for $\rho \res$ we resort to the peculiar extra condition. We collect this as yet another argument for prefering injective to projective objects in this context.

`Big' versions of Proposition~\ref{prop:coneperfectimpliesfullyfaithful} and its consequences now come for free.

\begin{corollary} \label{cor:fullyfaithfulonKcoac}
Let $\Lambda$ and $\Gamma$ be module finite over some regular commutative noetherian ring $k$, and the ring homomorphism $f$ be $k$-linear. Assume that $\Cone(f) \in \perf \Lambdae$.
%Suppose the conditions of Proposition~\ref{prop:coneperfectimpliesfullyfaithful} are met. 
Then the functor \[-\otimes_\Lambda \Gamma\colon \K_{\coac}(\Proj \Lambda)\to \K_{\coac}(\Proj \Gamma)\] is fully faithful. In particular, if additionally $\RHom_{\Lambda}(\Gamma, \Lambda) \in \perf \Gamma$ then the pair \[\left(\K_{\coac}(\Proj \Lambda), \Ker(\rho\res) \right)\] is a stable $t$-structure in $ \K_{\coac}(\Proj \Gamma)$.
\end{corollary}
\begin{proof}
Under these assumptions, $-\otimes_\Lambda \Gamma$ does give a functor between the categories in question by Proposition~\ref{prop restriction to acyclic}. Moreover, $-\otimes^{\mathbb L}_\Lambda \Gamma$ is fully faithful on the level of singularity categories by Proposition~\ref{prop:coneperfectimpliesfullyfaithful}. By Lemma~\ref{lem:compgenff} this suffices since the singularity category embeds as the compact objects in the homotopy category of coacyclic complexes of projective modules.
\end{proof}

\subsection*{Totally acyclic complexes}
We now descend further to the homotopy categories of totally acyclic complexes. As before, we investigate which of the above functors exist on this level under which conditions.
\begin{proposition} \label{prop:restrictiontotac}
Let $f\colon \Lambda \to \Gamma$ be a morphism of rings. Assume that $d = \pdim {}_{\Lambda}\Gamma < \infty$ and that $d' = \pdim \Gamma_{\Lambda} < \infty$
Then we have the following.
\begin{enumerate}[label=\emph{(\roman*)}]
\item The functor $\Hom_{\Lambda}(\Gamma,-)$ restricts to a functor between homotopy categories of totally acyclic complexes of injectives.
\item The functor $- \otimes_{\Lambda} \Gamma$ restricts to a functor between homotopy categories of totally acyclic complexes of projectives.
\item If additionally $\RHom_{\Lambda}(\Gamma, Q) \in \K^{\bounded}( \Proj \Gamma)$ for any $Q \in \Proj \Lambda$, then the functor $\rho \res$ restricts to a functor between categories of totally acyclic complexes of projectives.

If $\Gamma$ is perfect as $\Lambda$-module, the condition is satisfied when $\RHom_{\Lambda}(\Gamma, \Lambda)$ lies in $\perf \Gamma$.
\item If additionally $J \otimes_{\Lambda}^{\mathbb{L}} \Gamma \in \K^{\bounded}(\Inj \Gamma)$ for any $J \in \Inj \Lambda$, then the functor $\lambda \res$ restricts to a functor between categories of totally acyclic complexes of injectives.
\end{enumerate}
\end{proposition}

\begin{proof}
The first three points can be proved very similarly to Proposition~\ref{prop restriction to acyclic} above. We focus on (iv). Also as above, one sees that for a totally acyclic complex of injectives $I$, $\lambda \res I$ is acyclic again. For it to be totally acyclic, it remains to check that $\Hom_{\K(\Gamma)}(J, \lambda \res I) = 0$ for all injective $\Lambda$-modules $J$. By adjunction we have
\[ \Hom_{\K(\Lambda)}(J, \lambda \res I) = \Hom_{\K(\Gamma)}(\lambda( \widetilde\Omega^d J \otimes_{\Lambda} \Gamma ), I). \]
We observe that the left argument of the final $\Hom$-above is an injective resolution of $J \otimes_{\Lambda}^{\mathbb{L}} \Gamma$, so it is a bounded complex of injectives by assumption. Thus the total acyclicity of $I$ implies that these $\Hom$-groups are zero.
\end{proof}

By restriction of Corollary~\ref{cor:fullyfaithfulonKcoac} we immediately get the following. Notice that it is in fact not necessary to invoke Proposition~\ref{prop:coneperfectimpliesfullyfaithful} and Lemma~\ref{lem:compgenff} directly by assuming virtual Gorensteinness.

\begin{corollary}
\label{corsecondff}
Let $\Lambda$ and $\Gamma$ be module finite over some regular commutative noetherian ring $k$, and the ring homomorphism $f$ be $k$-linear. Assume that $\Cone(f) \in \perf \Lambdae$.
Then there is a fully faithful functor 
\[-\otimes_\Lambda \Gamma\colon \Ktac(\Proj \Lambda)\to \Ktac(\Proj \Gamma).\]
In particular, if additionally $\RHom_{\Lambda}(\Gamma, \Lambda) \in \perf \Gamma$ then the pair 
\[ \left(\Ktac(\Proj \Lambda), \Ker(\rho\res) \right)\]
is a stable $t$-structure in $ \Ktac(\Proj \Gamma)$.
\end{corollary}

\subsection*{Summary} Restricted to the case where $f\colon \Lambda\to\Gamma$ is a morphism of module finite algebras over some commutative noetherian Gorenstein ring such that $d=\maxx\{ \pdim {}_\Lambda \Gamma, \pdim\Gamma_\Lambda\}<\infty$, the results of the current section can be brought together in Diagram~\ref{bigpicture} below. A few comments are in order.

The top level is just a restatement of Proposition~\ref{prop:commutativeoverkwithadjoints}.

Evidently, the two lower levels are comprised of subcategories of the top level, so we have the vertical inclusions. Moreover, by Observation~\ref{obs:leftandrightadjoints} these inclusions have both left and right adjoints, as indicated by the downward pointing arrows labeled $\LL$ and $\RR$, respectively. It should be noted that these functors are given explicitly in restricted setups: The left adjoints are constructed in Bousfield's localization lemma --- see Corollary~\ref{rem:explicitL}. For the right adjoints, we give an explicit description in the case that the two rings are Artin algebras in Corollary~\ref{prop:explicitRforArtin}.

Most of the arrows, i.e.\ those not involving any $\RR$ or $\LL$, in the two lower levels of Diagram~\ref{bigpicture} arise by observing that the corresponding arrows on the top level restrict to these subcategories. In Propositions~\ref{prop restriction to acyclic} (middle level) and \ref{prop:restrictiontotac} (lower level) we collected which functors automatically restrict. The same propositions also contain information on more functors restricting under certain additional hypothesis, which gives a stronger version of Diagram~\ref{bigpicture} in a restricted setup.

Finally, on each of the two lower levels, we obtain one `extra adjoint' by composing adjoints around the square above the functor in question. For instance, the functor $\lambda \res \colon \K_{\ac}(\Inj \Gamma) \to \K_{\ac}(\Inj \Lambda)$ may be written as the composition
\[ \K_{\ac}(\Inj \Gamma) \hookrightarrow \K(\Inj \Gamma) \xrightarrow{\lambda \res} \K(\Inj \Lambda) \xrightarrow{\RR} \K_{\ac}(\Inj \Lambda), \]
whence it has a left adjoint given by the composition
\[ \K_{\ac}(\Inj \Lambda) \hookrightarrow \K(\Inj \Lambda) \xrightarrow{\lambda(\widetilde\Omega^d - \otimes_{\Lambda} \Gamma)} \K(\Inj \Gamma) \xrightarrow{\LL} \K_{\ac}(\Inj \Gamma). \]
Similarly we obtain all the extra adjoints involving $\LL$ or $\RR$.

\begin{equation}
\label{bigpicture}
\begin{tikzpicture}[baseline=-4.6cm]
 \node (KIA) at (-1,1) {$\K(\Inj \Lambda)$};
 \node (KIB) at (4,1) {$\K(\Inj \Gamma)$};
 \node (KPA) at (0,0) {$\K(\Proj \Lambda)$};
 \node (KPB) at (5,0) {$\K(\Proj \Gamma)$};
 \node (KcIA) at (-1,-4) {$\K_{\ac}(\Inj \Lambda)$};
 \node (KcIB) at (4,-4) {$\K_{\ac}(\Inj \Gamma)$};
 \node (KcPA) at (0,-5) {$\K_{\coac}(\Proj \Lambda)$};
 \node (KcPB) at (5,-5) {$\K_{\coac}(\Proj \Gamma)$};
 \node (KtIA) at (-1,-9) {$\Ktac(\Inj \Lambda)$};
 \node (KtIB) at (4,-9) {$\Ktac(\Inj \Gamma)$};
 \node (KtPA) at (0,-10) {$\Ktac(\Proj \Lambda)$};
 \node (KtPB) at (5,-10) {$\Ktac(\Proj \Gamma)$};
 \draw [->] (KIA) to node (KA) [fill=white,inner sep=0mm,outer sep=1.5cm] {$\approx$} (KPA);
 \draw [->] (KIB) to node (KB) [fill=white,inner sep=0mm,outer sep=1.5cm] {$\approx$} (KPB);
 \draw [->] (KcIA) to node (KcA) [fill=white,inner sep=0mm,outer sep=1.5cm] {$\approx$} (KcPA);
 \draw [->] (KcIB) to node (KcB) [fill=white,inner sep=0mm,outer sep=1.5cm] {$\approx$} (KcPB);
 \draw [->] (KtIA) to node (KtA) [fill=white,inner sep=0mm,outer sep=1.5cm] {$\approx$} (KtPA);
 \draw [->] (KtIB) to node (KtB) [fill=white,inner sep=0mm,outer sep=1.5cm] {$\approx$} (KtPB);
 \draw [->,bend left=60] (KIA) to node [above] {$\scriptstyle \lambda(\widetilde\Omega^d - \otimes_{\Lambda} \Gamma)$} (KIB);
 \draw [->,bend right] (KIB) to node [above] {$\scriptstyle \lambda \res$} (KIA);
 \draw [->] (KIA) to node [above] {$\scriptstyle \Hom_{\Lambda}(\Gamma, -)$} (KIB);
 \draw [->] (KPA) to node [above] {$\scriptstyle - \otimes_{\Lambda} \Gamma$} (KPB);
 \draw [->,bend left] (KPB) to node [above] {$\scriptstyle \rho \res$} (KPA);
 \draw [->,bend right=60] (KPA) to node [above=.5mm] {$\scriptstyle \rho \Hom_{\Lambda}(\Gamma, \widetilde\mho^d - )$} (KPB);
 \draw [->,bend left=60] (KcIA) to node [above] {$\scriptstyle \LL \lambda (\widetilde\Omega^d - \otimes_{\Lambda} \Gamma)$} (KcIB);
 \draw [->,bend right] (KcIB) to node [above] {$\scriptstyle \lambda \res$} (KcIA);
 \draw [->] (KcIA) to node [above] {$\scriptstyle \Hom_{\Lambda}(\Gamma, -)$} (KcIB);
 \draw [->] (KcPA) to node [above] {$\scriptstyle - \otimes_{\Lambda} \Gamma$} (KcPB);
 \draw [->,bend left] (KcPB) to node [above] {$\scriptstyle \RR \rho \res$} (KcPA);
 \draw [->,bend right] (KtIB) to node [above] {$\scriptstyle \LL \lambda \res$} (KtIA);
 \draw [->] (KtIA) to node [above] {$\scriptstyle \Hom_{\Lambda}(\Gamma, -)$} (KtIB);
 \draw [->] (KtPA) to node [above] {$\scriptstyle - \otimes_{\Lambda} \Gamma$} (KtPB);
 \draw [->,bend left] (KtPB) to node [above] {$\scriptstyle \RR \RR \rho \res$} (KtPA);
 \draw [->, bend right=10] (KA) to node [left] {$\scriptstyle \LL$} (KcA);
 \draw [right hook->] (KcA) to (KA);
 \draw [->, bend left=10] (KA) to node [right] {$\scriptstyle \RR$} (KcA);
 \draw [->, bend right=10] (KB) to node [left] {$\scriptstyle \LL$} (KcB);
 \draw [right hook->] (KcB) to (KB);
 \draw [->, bend left=10] (KB) to node [right] {$\scriptstyle \RR$} (KcB);
 \draw [->, bend right=10] (KcA) to node [left] {$\scriptstyle \LL$} (KtA);
 \draw [right hook->] (KtA) to (KcA);
 \draw [->, bend left=10] (KcA) to node [right] {$\scriptstyle \RR$} (KtA);
 \draw [->, bend right=10] (KcB) to node [left] {$\scriptstyle \LL$} (KtB);
 \draw [right hook->] (KtB) to (KcB);
 \draw [->, bend left=10] (KcB) to node [right] {$\scriptstyle \RR$} (KtB);
\end{tikzpicture}
\end{equation}

\section{Approximations of modules} \label{section approximationsofmodules}

In \cite{MR1859047} and the subsequent \cite{MR2283103}, Jørgensen plays the following game in order to obtain right approximations by Gorenstein projective modules. Under rather forgiving hypotheses on $\Lambda$, he establishes the contravariant finiteness of $\Ktac(\Proj \Lambda)$ in $\K(\Proj \Lambda)$ before passing to $\Modules \Lambda$ by taking boundaries, a procedure which turns out to take a right $\Ktac(\Proj \Lambda)$-approximation of a projective resolution $\rho M$ of $M$ to a right $\GProj \Lambda$-approximation of $M$. We will now follow these ideas in two directions: We first observe that over an arbitrary ring $R$, any contravariantly finite subcategory of $\K(\Proj R)$ gives rise to a contravariantly finite subcategory of $\Modules R$ by taking cokernels. Next, an ever so slight expansion of the above Theorem~\ref{thm:jorgensenGProjcontfinite} is obtained.

\begin{lemma} \label{lem:preimageofcontfiniteiscontfinite}
Let $\A$ be a Frobenius exact category. If $\underline \B$ is a contravariantly finite subcategory in $\underline\A$, then its preimage $\B$ is contravariantly finite in $\A$.
\end{lemma}
\begin{proof}
Each $X \in \A$ admits a right $\B$-approximation
\[
B \oplus P_X \xrightarrow{\big(\begin{smallmatrix} \pi & p \end{smallmatrix}\big)} X
\]
where $p$ is a projective cover and $\pi$ is any lift of a right $\underline\B$-approximation.
\end{proof}

\begin{proposition} \label{prop:contfiniteboundaries}
Let $R$ be a ring, and suppose $\SSS$ is a contravariantly finite subcategory of $\K(\Proj R)$. Denote by $\T$ the preimage of $\SSS$ in $\C(\Proj R)$ and by $\Cok(\T)$ the subcategory of $\Modules R$ consisting of all the cokernels of the form $\Cok(T^{-1}\to T^0)$ for $T$ a complex belonging to $\T$. Then $\Cok(\T)$ is contravariantly finite in $\Modules R$.
\end{proposition}
\begin{proof}
For each $X\in\C(\Proj R)$ we denote by $\Cok(X)=\Cok(X^{-1}\to X^0)\in \Modules R$ and let $\pi_X\colon X\to \Cok(X)$ be the natural projection. Pick $M\in\Modules R$ and choose a projective resolution $\rho M$. In particular $M=\Cok(\rho M)$. Since $\T$ is contravariantly finite in $\C(\Proj R)$ by Lemma~\ref{lem:preimageofcontfiniteiscontfinite}, $\rho M$ admits a right $\T$-approximation $\alpha\colon t(\rho M) \to \rho M$. We claim that the induced morphism $a\colon\Cok(t(\rho M))\to M$ satisfying
\[
a \circ \pi_{t(\rho M)}=\pi_{\rho M} \circ \alpha
\]
is a right $\Cok(\T)$-approximation in $\Modules R$. What we need to show is that each morphism $b \colon \Cok(T) \to M$ in $\Modules R$ with $T\in\T$ factors through $a$. By `comparison', such a $b$ lifts to a chain map $\beta\colon T\to\rho M$ such that
\[
b \circ \pi_T = \pi_{\rho M} \circ \beta,
\]
and since $\alpha$ is a right $\T$-approximation, there is a chain map $\gamma\colon T \to t(\rho M)$ making
\[
\beta = \alpha \circ \gamma.
\]
Letting $c\colon \Cok(T)\to\Cok(t(\rho M))$ be the induced morphism in $\Modules R$ ensuring
\[
c \circ \pi_T = \pi_{t(\rho M)} \circ \gamma,
\]
we have by now established the diagram
\begin{center}
\begin{tikzpicture}[baseline=(current bounding box.center)]
   \matrix(a)[matrix of math nodes,
   row sep=2em, column sep=2em,
   text height=1.5ex, text depth=0.25ex]
   {t(\rho M) & \Cok(t(\rho M)) & \Cok(T) & T  \\
    \rho M & M && \rho M \\};
   \path[->,font=\scriptsize] (a-1-1) edge node[below]{$\pi_{t(\rho M)}$}(a-1-2)
   (a-1-3) edge node[above]{$c$}(a-1-2)
   (a-1-4) edge node[below]{$\pi_T$}(a-1-3)
   (a-2-1) edge node[above]{$\pi_M$}(a-2-2)
   (a-2-4) edge node[above]{$\pi_M$}(a-2-2)
   (a-1-1) edge node[left]{$\alpha$}(a-2-1)
   (a-1-2) edge node[left]{$a$}(a-2-2)
   (a-1-4) edge node[left]{$\beta$}(a-2-4)
   (a-1-3) edge node[above]{$b$}(a-2-2)
   (a-1-4) edge[out=160, in=20] node[above]{$\gamma$}(a-1-1)
   (a-2-4) edge[out=200, in=340] node[below]{$\id$}(a-2-1);
\end{tikzpicture}
\end{center}
with each square commutative. It follows that $b \circ \pi_T = a \circ c \circ \pi_T$ and hence the desired factorization $b = a\circ c$ since $\pi_T$ is an epimorphism.
\end{proof}

The alluded to existence of right approximations by Gorenstein projective modules is now immediate once one makes the following standard observation.

\begin{lemma} \label{lem:rightadjointimpiescontfinite}
Let $\A$ be a category and suppose $\B$ is a subcategory such that the inclusion $\mu \colon \B \hookrightarrow \A$ has a right adjoint $\mu_{\ast}$. Then $\B$ is contravariantly finite in $\A$.
\end{lemma}
\begin{proof}
It suffices to show that, for each $A\in\A$, the counit $\varepsilon_A\colon \mu \circ\mu_{\ast}(A)\to A$ is a right $\B$-approximation. But the isomorphism of the adjunction is given, for each $B\in \B$, by
\[
\B(B, \mu_{\ast}(A)) \xrightarrow{\mu} \A(\mu(B), \mu\circ\mu_{\ast}(A)) \xrightarrow{\A(\mu(B),\varepsilon_A)} \A(\mu(B), A).
\]
In particular this means that $\A(\mu(B),\varepsilon_A)$ is an epimorphism.
\end{proof}

\begin{corollary} \label{cor:GProjcontfinite}
Let $\Lambda$ be a noetherian ring with a dualizing complex. Then $\GProj \Lambda$ is contravariantly finite in $\Modules \Lambda$.
\end{corollary}
\begin{proof}
A combination of Observation~\ref{obs:leftandrightadjoints} and Lemma~\ref{lem:rightadjointimpiescontfinite} reveals the contravariant finiteness of $\Ktac(\Proj \Lambda)$ in $\K(\Proj \Lambda)$, so Proposition~\ref{prop:contfiniteboundaries} settles the claim.
\end{proof}

\section{An explicit right adjoint} \label{section:rightadjoint}

The aim of this section is to provide a somewhat strengthened dual of Bousfield's localization lemma \cite{MR543337,MR728454}. We start by recalling the construction as presented in \cite{MR1191736}, see also \cite[Chapter III, Theorem 2.3]{MR2327478} and \cite[Theorem~4.3]{MR2927802}.

Here, for the sake of generality, we consider perpendicular categories with respect to only certain shifts, rather than all shifts. More precisely, we write
\begin{align*}
\X^{\perp_0} & = \{ T \in \T \mid \T(\X, T) = 0 \}, & \X^{\perp_{\leq 0}} & = \{ T \in \T \mid \forall i \leq 0 \colon \T(\X, T[i]) = 0 \}, \\
\prescript{\perp_0}{} \X & = \{ T \in \T \mid \T(T,\X) = 0 \}, & \prescript{\perp_{\leq 0}}{} \X & = \{ T \in \T \mid \forall i \leq 0 \colon  \T(T,\X[i]) = 0 \},
\end{align*}
and similar for \( \perp_{\geq 0} \).
With this convention we obviously have
\begin{equation} \label{eqallshifts}
\X^{\perp} = \{ X[i] \mid X \in \X, i \in \mathbb{Z} \}^{\perp_0}
\end{equation}
and similar for left perpendicular subcategories.

\begin{theorem} \label{thm:boudfieldslemma}
Let $\T$ be a triangulated category which has coproducts, and let $\X$ be a set of compact objects.
\begin{enumerate}[label=\emph{(\roman*)}]
\item The pair
\[ \left( \prescript{\perp_0}{}(\X^{\perp_0}), \X^{\perp_0}\right) \]
is a torsion pair on $\T$.
\item The pair
\[ \left( \prescript{\perp_{\leq 0}}{}(\X^{\perp_{\leq 0}}), \X^{\perp_{\leq 0}} \right) \] 
is a t-structure, and moreover \(  \prescript{\perp_{\leq 0}}{}(\X^{\perp_{\leq 0}}) \) is the smallest subcategory of $\T$ containing $\X$ and closed under \([1]\), coproducts, and extensions.
\end{enumerate}
\end{theorem}

From the proof of this theorem, we only recall the following explicit construction of the triangle
\[
T_{\X} \to T \to T_{\X^{\perp_0}} \to T_{\X}[1]
\]
with $T_{\X} \in \prescript{\perp_0}{}(\X^{\perp_0})$ and $T_{\X^{\perp_0}} \in \X^{\perp_0}$:

Let $T_0=T$. Pick a right $\Add \X$-approximation $X_0 \to T_0$ of $T_0$ --- this approximation exists since $\X$ is a set and $\T$ has coproducts --- and complete it to a triangle $X_0 \to T_0 \to T_1 \to X_0[1]$. Pick a right $\Add \X$-approximation $X_1 \to T_1$ of $T_1$ and complete it to a triangle $X_1 \to T_1 \to T_2 \to X_1[1]$. Then $T_{\X^{\perp_0}} = \hocolim T_i$, and $T_{\X}$ is obtained by completing the natural map $T \to T_{\X^{\perp_0}}$ to a triangle.

This construction covers each functor $\LL$ appearing in Diagram~\ref{bigpicture} in the Summary of Section~\ref{sec:bigcat}. In that setup we consider subcategories $X^{\perp}$ orthogonal to a given object with respect to all extensions as specified in Proposition~\ref{prop:dualizingtestscoacyclicity}. By \eqref{eqallshifts} this is a special case of the setup of Theorem~\ref{thm:boudfieldslemma}, and in this case we obtain a stable t-structure $( \prescript{\perp}{}(\X^{\perp}), \X^{\perp} )$. Thus we obtain the following descriptions of the left adjoints.

%Indeed, in the notation of Theorem~\ref{thm:boudfieldslemma}, by Proposition~\ref{prop:dualizingtestscoacyclicity} we have the following.

\begin{corollary} \label{rem:explicitL}
Let $\Lambda$ be a noetherian ring with a dualizing complex $D_\Lambda$.
Then the following hold.
\begin{enumerate}[label=\emph{(\roman*)}]
\item The left adjoint of the inclusion $\K_{\coac}(\Proj \Lambda) \hookrightarrow \K(\Proj \Lambda)$ is given by
\[
T \mapsto T_{\left(\rho\RHom_\Lambda(D_\Lambda,\Lambda)\right)^{\perp}};
\]
\item The left adjoint of the inclusion $\K_{\ac}(\Inj \Lambda) \hookrightarrow \K(\Inj \Lambda)$ is given by
\[
T \mapsto T_{(\lambda\Lambda)^{\perp}};
\]
\item The left adjoint of the inclusion $\Ktac(\Proj \Lambda) \hookrightarrow \K(\Proj \Lambda)$ is given by
\[
T \mapsto T_{\left((\rho\RHom_\Lambda(D_\Lambda,\Lambda)) \oplus \Lambda\right)^{\perp}};
\]
\item The left adjoint of the inclusion $\Ktac(\Inj \Lambda) \hookrightarrow \K(\Inj \Lambda)$ is given by
\[
T \mapsto T_{((\lambda\Lambda) \oplus D_\Lambda)^{\perp}}.
\]
\end{enumerate}
\end{corollary}
%Goes in Section 6(?): the subcategory ${}^{\perp}\X$ is prone to technical difficulties that do not apply to that of $\X^{\perp}$. In broad strokes, problems arise because we have no way of pulling products, or more generally homotopy limits, out from the first argument of $\Hom$. In other words, the typical category of interest to us will often not have any non-zero cocompact objects.

Unfortunately, the naive dual of Theorem~\ref{thm:boudfieldslemma} is not as useful as the original in the situations we are typically interested in. Indeed, triangulated categories which somehow derive from module categories rarely have any non-zero cocompact objects. However, as we will point out in the sequel, we can get away with the weaker notion of $0$-cocompactness which is not as elusive in concrete situations.

This section consists of two essentially independent parts. First, in Theorem~\ref{theorem.t_from_left_perp} we provide the desired dual. Then, in Theorem~\ref{theorem.is_0_cocomp} and in particular Corollary~\ref{corollary.artinina_0-cocompact} we show that the prerequisites of our dual are met in reasonable situations.

We start by introducing our notion of $0$-cocompactness.

\begin{definition}
A sequence
\[ A_0 \xrightarrow{f_0} A_1 \xrightarrow{f_1} A_2 \to \cdots \]
of objects and morphisms in an abelian category is \emph{dual Mittag-Leffler} if for each $i$ the increasing sequence
\[ 0 \subset \Ker f_i \subset \Ker f_{i+1} f_i \subset \cdots \]
stabilizes.

An object $X$ in a triangulated category $\T$ is \emph{$0$-cocompact} if $\T(\holim Y_i,X)=0$ for each sequence
\[ \cdots \to Y_2 \to Y_1 \to Y_0  \]
in $\T$ with the property that the induced sequence
\[ \T(Y_0,X[1]) \to \T(Y_1,X[1]) \to \T(Y_2,X[1]) \to \cdots \]
is dual Mittag-Leffler and $\colim \T(Y_i,X)=0$.
\end{definition}

\begin{remark}
Let us motivate the name ``0-cocompact'' and explain in what way it is a weak version of cocompact.

Recall that in a triangulated category $\T$ with set-indexed coproducts, an object $X$ is \emph{compact} if the functor $\T(X,-)\colon \T\to \Ab$ commutes with coproducts. If $\T$ has set-indexed products, then the dual notion should be formulated as follows. An object $X\in\T$ is `cocompact' if the functor $\T(-,X)\colon \T^{\op}\to \Ab$ turns products to coproducts, i.e.\ for any set of objects $\left(Y_i \ | \ i\in I\right)$ in $\T$ we have an isomorphism
\begin{align}\label{align cocompact isomorphism}
\Hom_{\T}(\prod_{i\in I}Y_i,X)\cong \coprod_{i\in I}\Hom_{\T}(Y_i,X).
\end{align}
Consider now a finite dimensional algebra $\Lambda$. Clearly, $\Lambda$ is a compact object in the derived category $\D(\Modules \Lambda)$. On the other hand, the dual $D\Lambda$ awkwardly fails to be cocompact in $\D(\Modules \Lambda)$. However, we show in Theorem~\ref{theorem.is_0_cocomp} that $D\Lambda$ is in fact $0$-cocompact in the homotopy category $\K(\Modules \Lambda)$.

Replacing products by homotopy limits and coproducts by colimits, $X$ being $0$-cocompact simply means that if the right hand side of (\ref{align cocompact isomorphism}) is zero and some extra conditions hold, then the left hand side of (\ref{align cocompact isomorphism}) is also zero. As it turns out, this property suffices for our purposes. Let us revisit the cocomplete case and consider a sequence $X_0 \to X_1 \to X_2 \to \cdots $ in $\T$. Then by Neeman~\cite{MR1191736}, we know that there is the following isomorphism as functors on $\T^{\mathsf{c}}$.
\begin{equation}
\label{isomcolimcompactcase}
\colim \T(-, X_i) \cong \T(-, \hocolim{X_i})
\end{equation}
From this point of view, the $0$-cocompactness of $X$ means that the vanishing of the dual left side of $(\ref{isomcolimcompactcase})$ implies the vanishing of the dual right side of (\ref{isomcolimcompactcase}).
\end{remark}

\begin{remark}
The above definition might seem quite unnatural: Why would one want to require the dual Mittag-Leffler condition at all? And why for shifts?

Consider the case that \( \T = \D(\mathcal{A}) \), where \( \mathcal{A} \) is an abelian category with exact products. For a sequence
\[ \cdots \to Y_2 \to Y_1 \to Y_0 \]
of complexes, we see that the defining triangle
\[ \holim Y_i \to \prod Y_i \to \prod Y_i \to \]
gives rise to the exact sequence
\[ \begin{tikzpicture}[xscale=1.3]
\node (A) at (0,0) {\(\prod \HH^0(Y_i[-1])\)};
\node (B) at (2,0) {\(\prod \HH^0(Y_i[-1])\)}; 
\node (C) at (4,0) {\( \HH^0(\holim Y_i) \)}; 
\node (D) at (6,0) {\(\prod \HH^0(Y_i)\)}; 
\node (E) at (8,0) {\(\prod \HH^0(Y_i)\)};
\node (L1) at (3,-1) {\( \varprojlim^1 \HH^0(Y_i[-1])\)}; 
\node (L) at (5,-1) {\( \varprojlim \HH^0(Y_i) \) };
\draw [->] (A) to (B);
\draw [->] (B) to (C);
\draw [->] (C) to (D);
\draw [->] (D) to (E);
\draw [->>] (B) to (L1);
\draw [>->] (L1) to (C);
\draw [->>] (C) to (L);
\draw [>->] (L) to (D);
\end{tikzpicture} \]
Thus we see that, even if a \( \Hom \)-functor behaves as nicely as homology, we can only get a result on \( \holim Y_i \) provided we have a condition on the derived projective limit of the shifted sequence.

To make this more explicit, take \( \mathcal{A} \) to be the category of vector spaces over a field \( k \), and choose \( X = k \). Then \( \Hom_{\D(k)}(-, k) = D \HH^0 \). Thus we have the short exact sequence
\[ 0 \to \Hom_{\D(k)}( \varprojlim Y_i, k) \to \Hom_{\D(k)} ( \holim Y_i, k) \to  \Hom_{\D(k)}( {\varprojlim}^1 Y_i, k[1]) \to 0 \]
showing that we need to make both end terms vanish in order to conclude that the middle also vanishes.
\end{remark}

We are now ready to state and prove the main result of this section.

\begin{theorem} \label{theorem.t_from_left_perp}
Let $\T$ be a triangulated category which has products. Let $\X$ be a set of $0$-cocompact objects.
\begin{enumerate}[label=\emph{(\roman*)}]
\item The pair
\[ 
\left(\prescript{\perp_{\geq 0}}{}{\X}, (\prescript{\perp_{\geq 0}}{}{\X})^{\perp_{\geq 0}}\right) 
\]
is a co-t-structure in $\T$.

\item The pair
\[ 
\left(\prescript{\perp}{}{\X}, (\prescript{\perp}{}{\X})^{\perp}\right) 
\]
is a stable t-structure, and moreover $(\prescript{\perp}{}{\X})^{\perp}$ is the smallest triangulated subcategory of \( \T \) which contains \( \X \) and is closed under products.
\end{enumerate}
\end{theorem}

\begin{proof}
In (i), the pair is clearly $\Hom$-orthogonal. It is also immediate that \( \prescript{\perp_{\geq 0}}{}{\X} \) is closed under \( [-1] \) and that \(  (\prescript{\perp_{\geq 0}}{}{\X})^{\perp_{\geq 0}} \) is closed under \( [1] \). So it remains to show that each $T\in\T$ appears in a triangle
\begin{equation}\label{equation functorial triangle}
  T_{\prescript{\perp_{\geq 0}}{}{\X}} \to T \to T_{\X} \to T_{\prescript{\perp_{\geq 0}}{}{\X}}[1]
\end{equation}
with $T_{\prescript{\perp_{\geq 0}}{}{\X}}\in\prescript{\perp_{\geq 0}}{}{\X} $ and $T_{\X}\in(\prescript{\perp_{\geq 0}}{}{\X})^{\perp_{\geq 0}}$.

Consider the following construction: Write $T_0 = T$, and \( \X[\geq 0] = \{ X[i] \mid X \in \X \text{ and } i \geq 0 \} \). Pick a left $\Prod \X[\geq 0]$-approxima\-tion $T_0 \to X_0$ of $T_0$ --- which exists by Lemma~\ref{lem:covariantlyfinite} --- and complete it to a triangle
\begin{equation}
\label{firsttriangle}
T_1 \to T_0 \to X_0 \to T_1[1].
\end{equation}
Pick a left $\Prod \X[\geq 0]$-approximation $T_1 \to X_1$ of $T_1$, complete it to a triangle
\begin{equation}
\label{secondtriangle}
T_2 \to T_1 \to X_1 \to T_2[1],
\end{equation}
and iterate in this fashion.
Then put
\[
T_{\prescript{\perp_{\geq 0}}{}{\X}} = \holim T_i \text{ and } T_{\X} = \Cone[ T_{\prescript{\perp_{\geq 0}}{}{\X}} \to T].
\]
It remains for us to show that these objects $T_{\prescript{\perp_{\geq 0}}{}{\X}}$ and $T_{\X}$ belong to the subcategories $\prescript{\perp_{\geq 0}}{}{\X}$ and $(\prescript{\perp_{\geq 0}}{}{\X})^{\perp_{\geq 0}}$, respectively.

It follows by construction that each map in the sequence
\[ \T(T_0, \X[i]) \to \T(T_1, \X[i]) \to \T(T_2, \X[i]) \to \cdots \]
vanishes for all positive \( i \). Indeed, for any map $T_j\to X'[i]$ with $X'$ in $\X$, the composition $T_{j+1} \to T_j \to X'[i]$ is zero since the map $T_j \to X'[i]$ factors through $X_j$ and by the triangle $(\ref{firsttriangle})$ the composition $T_{j+1} \to T_j \to X_j$ is zero. This implies that the above sequences are dual Mittag-Leffler and have vanishing colimit. Since the objects in $\X$ are $0$-cocompact, it follows that $T_{\prescript{\perp_{\geq 0}}{}{\X}} = \holim T_i$ lies in $\prescript{\perp_{\geq 0}}{}{\X}$.

The proof that $T_{\X} \in (\prescript{\perp_{\geq 0}}{}{\X})^{\perp_{\geq 0}}$ is almost verbatim the same as steps (ii) to (v) in the proof of \cite[Proposition~4.5]{MR2927802}. This completes the proof of (i).

The fact that the pair in (ii) is a stable t-structure follows immediately from (i): Either apply (i) to the set of \(0\)-cocompacts \( \X[\leq 0] = \{ X[i] \mid X \in \X \text{ and } i \leq 0 \} \), or replace any occurrence of \( X[\geq 0] \) in the proof of (i) by \( X[\mathbb{Z}] = \{ X[i] \mid X \in \X \text{ and } i \in \mathbb{Z} \} \).

For the proof of the final claim let us denote by $\widehat{\X}$ the smallest triangulated subcategory of $\T$ containing $\X$ and closed under products. It is immediate that $\widehat{\X} \subseteq (\prescript{\perp}{}{\X})^{\perp}$, since the latter category it triangulated, contains $\X$, and is closed under products.

To see the converse inclusion, if suffices to show that $T_{\X} \in \widehat{\X}$ for any $T \in \T$. We write $Y_i$ for the cone of the composed map $T_i \to T$. In particular $Y_0 = 0$ and $Y_1 = X_0$. By the octahedral axiom the $Y_i$ also appear (iteratively) in triangles
\[
X_i \to Y_{i+1} \to Y_i \to X_i[1],
\]
and in particular $Y_i$ is in the smallest triangulated subcategory containing $\X$.

Now we take the product of all the triangles defining the $Y_i$, that is the triangle
\[ \prod_{i = 0}^{\infty} T_i \to \prod_{i=0}^{\infty} T \to \prod_{i=0}^{\infty} Y_i \to \prod_{i = 0}^{\infty} T_i [1]. \]
Observe that the solid square in the following diagram commutes, where the vertical maps between products are the ones defining the homotopy limits. (That is the left vertical map is $1 - s$, where $s$ is given by the morphisms connecting the $T_i$, and the middle vertical map is $1-s'$, where $s'$ just sends each factor $T$ identically to the previous one.). Hence, by the triangulated $3 \times 3$-lemma the diagram may be completed to a commutative one with triangles in all rows and columns.
\[ \begin{tikzpicture}
 \matrix (M) [matrix of math nodes, row sep=2em, column sep=2em, text height=1.5ex, text depth=0.25ex] {
   T_{\prescript{\perp_0}{}{\X}} & T & T_{\X} & T_{\prescript{\perp_0}{}{\X}}[1] \\
   \prod_{i = 0}^{\infty} T_i  & \prod_{i=0}^{\infty} T &  \prod_{i=0}^{\infty} Y_i &  \prod_{i = 0}^{\infty} T_i [1] \\
   \prod_{i = 0}^{\infty} T_i  & \prod_{i=0}^{\infty} T &  \prod_{i=0}^{\infty} Y_i &  \prod_{i = 0}^{\infty} T_i [1] \\
   T_{\prescript{\perp_0}{}{\X}}[1] & T[1] & T_{\X}[1] \\
  };
  \draw [dashed,->] (M-1-1) to (M-1-2);
  \draw [dashed,->] (M-1-2) to (M-1-3);
  \draw [dashed,->] (M-1-3) to (M-1-4);
  \draw [->] (M-2-1) to (M-2-2);
  \draw [->] (M-2-2) to (M-2-3);
  \draw [->] (M-2-3) to (M-2-4);
  \draw [->] (M-3-1) to (M-3-2);
  \draw [->] (M-3-2) to (M-3-3);
  \draw [->] (M-3-3) to (M-3-4);
  \draw [dashed,->] (M-4-1) to (M-4-2);
  \draw [dashed,->] (M-4-2) to (M-4-3);
  \draw [->] (M-1-1) to (M-2-1);
  \draw [->] (M-1-2) to (M-2-2);
  \draw [->,dashed] (M-1-3) to (M-2-3);
  \draw [->] (M-1-4) to (M-2-4);
  \draw [->] (M-2-1) to node [left] {$1-s$} (M-3-1);
  \draw [->] (M-2-2) to node [left] {$1-s'$} (M-3-2);
  \draw [->,dashed] (M-2-3) to (M-3-3);
  \draw [->] (M-2-4) to (M-3-4);
  \draw [->] (M-3-1) to (M-4-1);
  \draw [->] (M-3-2) to (M-4-2);
  \draw [->,dashed] (M-3-3) to (M-4-3);
\end{tikzpicture} \]
Note that the triangle $T \to \prod_{i=0}^{\infty} T \overset{1-s'}{\to} \prod_{i=0}^{\infty} T \to T[1]$ splits. Therefore the first dashed horizontal map is the natural map $\holim T_i \to T$, and the top triangle is the one defining $T_{\X}$. Now the vertical dashed triangle shows that $T_{\X}$ is an extension of $\prod_{i = 0}^{\infty}Y_i[-1]$ and $\prod_{i = 0}^{\infty}Y_i$, hence in $\widehat{\X}$.
\end{proof}

\begin{remark}
There is an arguably more conceptual proof of (ii) above in case $\T$ has an enhancement: Assume $\T$ sits at the base of a stable derivator. Then one may consider the triangles
\[ T_i \to T \to Y_i \to T_i \]
and obtain a new triangle taking the degree-wise homotopy limit
\[ \holim T_i \to T \to \holim Y_i \to \holim T_i [1]. \]
This is true by dual of \cite[Proposition A.5, Corollary A.6]{MR3031826} --- note that one may have to adjust the maps connecting the $Y_i$. Now $\holim Y_i$ is easily seen to lie in $\widehat{\X}$.
\end{remark}

We now provide examples of $0$-cocompact objects.

\begin{theorem} \label{theorem.is_0_cocomp}
Let $\Lambda$ be an algebra over a commutative artinian ring $k$ and denote by $D = \Hom_{k}(-, E)$ the standard duality. For any finite complex of finitely generated left $\Lambda$-modules $X$, the complex $DX$ of right $\Lambda$-modules is $0$-cocompact in $\K(\Modules \Lambda)$.
\end{theorem}
\begin{proof}
We denote by $F$ the functor given by taking zeroth homology of the total tensor product with $X$, i.e.\ $F = \HH^0( - \otimes_{\Lambda}^{\tot} X)$, and observe that
\begin{align*}
\Hom_{\K}(-, DX) & = \HH^0( \HOM_{\Lambda}( -, DX)) = \HH^0( D(- \otimes_{\Lambda}^{\tot} X) ) \\
& = D \HH^0 (- \otimes_{\Lambda}^{\tot} X) = D F.
\end{align*}
Moreover, since $X$ is finitely generated, the tensor product and hence also $F$ commute with products.
%After these preparations we are ready to verify the $0$-cocompactness of $DX$. Let
%\[ \cdots \to T_2 \to T_1 \to T_0 \]
%be a sequence of morphisms and objects in $\K(\Modules \Lambda)$ such that
%\[ \Hom_{\K}(T_0, DX) \to \Hom_{\K}(T_1, DX) \to \Hom_{\K}(T_2, DX) \to \cdots \]
%is dual Mittag-Leffler and has vanishing colimit. Dualizing, and using the fact that $D$ turns colimits into limits, we obtain that
%\[ \cdots \to D \Hom_{\K}(T_2, DX) \to D \Hom_{\K}(T_1, DX) \to D \Hom_{\K}(T_0, DX) \]
%has vanishing limit, i.e.\ $\limit D^2 F(T_i) = 0$. Since limits are left exact it follows that also $\limit F(T_i) = 0$.
%
%On the other hand, since the sequence above is dual Mittag-Leffler, and is the dual of the sequence
%\[ \cdots \to F(T_2) \to F(T_1) \to F(T_0), \]
%it follows that the latter sequence satisfies the Mittag-Leffler condition, and in particular that $\limit^1 F(T_i)=0$.
%
%These two properties together mean that the morphism
%\[ \prod F(T_i) \xrightarrow{1 - (f_i)} \prod F(T_i) \]
%is an isomorphism. Since $F$ commutes with products, it follows that $F(\holim T_i) = 0$, and thus also
%\[ \Hom_{\K}(\holim T_i, DX) = D F(\holim T_i) = 0. \qedhere \]

After these preparations we are ready to verify the $0$-cocompactness of $DX$. Let
\[ \cdots \to T_2 \to T_1 \to T_0 \]
be a sequence of morphisms and objects in $\K(\Modules \Lambda)$ such that
\[ \Hom_{\K}(T_0, DX[1]) \to \Hom_{\K}(T_1, DX[1]) \to \Hom_{\K}(T_2, DX[1]) \to \cdots \]
is dual Mittag-Leffler and that \( \colim \Hom_{\K}(T_i, DX) = 0 \).

Dualizing this second assumption, and using the fact that $D$ turns colimits into limits, we obtain that
\[ \cdots \to D \Hom_{\K}(T_2, DX) \to D \Hom_{\K}(T_1, DX) \to D \Hom_{\K}(T_0, DX) \]
has vanishing limit, i.e.\ $\limit D^2 F(T_i) = 0$. Since limits are left exact it follows that also $\limit F(T_i) = 0$.  In particular the morphism
\[  \prod F(T_i) \xrightarrow{1 - (f_i)} \prod F(T_i) \]
is monic.

On the other hand, note that the sequence in the first assumption is the dual of
\[ \cdots \to F(T_2[-1]) \to F(T_1[-1]) \to F(T_0[-1]). \]
It follows that the latter sequence satisfies the Mittag-Leffler condition, and in particular that $\limit^1 F(T_i[-1])=0$. Thus the map
\[ \prod F(T_i[-1]) \xrightarrow{1 - (f_i)} \prod F(T_i[-1]) \]
is epic.

Since $F$ commutes with products, the exact sequence
\[  F(\prod T_i[-1]) \to F(\prod T_i[-1])  \to F(\holim T_i) \to  F(\prod T_i) \to F(\prod T_i) \]
starts with an epimorphism and ends with a monomorphism, showing that the middle term vanishes. Thus
\[ \Hom_{\K}(\holim T_i, DX) = D F(\holim T_i) = 0. \qedhere \]
\end{proof}

\begin{remark}
In the above theorem, it is not relevant that $\Lambda$ is concentrated in degree $0$, i.e.\ the theorem still holds for $\Lambda$ a dg $k$-algebra.
\end{remark}

However, here we are actually most interested in an even more special case:

\begin{corollary} \label{corollary.artinina_0-cocompact}
Let $\Lambda$ be an Artin algebra. Then any finite complex of finitely generated $\Lambda$-modules is $0$-cocompact in $\K(\Modules \Lambda)$.
\end{corollary}

\begin{proof}
This follows directly from Theorem~\ref{theorem.is_0_cocomp}, using only the oberservation that the duals of finitely generated left $\Lambda$-modules are precisely the finitely generated right $\Lambda$-modules.
\end{proof}

What is more, for Artin algebras, Corollary~\ref{prop:explicitRforArtin} below says that each functor $\RR$ appearing in Diagram~\ref{bigpicture} in the Summary of Section~\ref{sec:bigcat} is covered by the explicit construction in (the proof of) Theorem~\ref{theorem.t_from_left_perp}. As usual, $D_\Lambda$ denotes the dualizing complex for $\Lambda$ --- see Example~\ref{ex:dualizingforartin}. 

By Propostion~\ref{prop:leftperpforartin}, the relevant categories are given as left orthogonal to certain objects. Our first task is to make sure that these objects are $0$-cocompact.

\begin{lemma} \label{lem:four0cocompactsforArtin}
Let $\Lambda$ be an Artin algebra. Then
\begin{enumerate}[label=\emph{(\roman*)}]
\item the complex $\Lambda$ is $0$-cocompact in $\K(\Proj \Lambda)$;
\item the complex $D_\Lambda$ is $0$-cocompact in $\K(\Inj \Lambda)$;
\item the complex $\rho D_\Lambda$ is $0$-cocompact in $\K(\Proj \Lambda)$;
\item the complex $D_\Lambda \otimes_\Lambda^{\mathbb L} D_\Lambda$ is $0$-cocompact in $\K(\Inj \Lambda)$.
\end{enumerate}
\end{lemma}
\begin{proof}
(i) and (ii) are immediate from Corollary~\ref{corollary.artinina_0-cocompact}. In order to show (iii), let
\[
 \cdots \to X_2 \to X_1 \to X_0
\]
be a sequence in $\K(\Proj \Lambda)$ such that the induced sequence
\[
\Hom_{\K(\Proj\Lambda)}(X_0, \rho D_\Lambda[1]) \to \Hom_{\K(\Proj\Lambda)}(X_1, \rho D_\Lambda[1]) \to \cdots %\Hom_{\K(\Proj\Lambda)}(X_2, \rho D_\Lambda) \to \cdots
\]
is dual Mittag-Leffler and $\colim\Hom_{\K(\Proj\Lambda)}(X_i, \rho D_\Lambda)=0$.  
%has vanishing colimit. 
Since $\rho$ is right adjoint to the inclusion $\K(\Proj\Lambda)\hookrightarrow\K(\Modules \Lambda)$, the sequence
\[
\Hom_{\K(\Modules\Lambda)}(X_0, D_\Lambda[1]) \to \Hom_{\K(\Modules\Lambda)}(X_1, D_\Lambda[1]) \to \cdots %\Hom_{\K(\Modules\Lambda)}(X_2, D_\Lambda[1]) \to \cdots
\]
is again dual Mittag-Leffler and the colimit of $\Hom_{\K(\Modules\Lambda)}(X_i, D_\Lambda)$ vanishes.
%with vanishing colimit. 
The $0$-cocompactness of $D_\Lambda$ in $\K(\Modules \Lambda)$ thus implies the $0$-cocompactness of $\rho D_\Lambda$ in $\K(\Proj \Lambda)$. Indeed,
\[
\Hom_{\K(\Proj \Lambda)}(\holim X_i, \rho D_\Lambda) \cong \Hom_{\K(\Modules \Lambda)}(\holim X_i, D_\Lambda) =0
\]
since $\Proj \Lambda$ is closed under products. Claim (iv) now follows from the fact that the equivalence $-\otimes_\Lambda D_\Lambda \colon\K(\Proj\Lambda)\to\K(\Inj\Lambda)$ preserves $0$-cocompactness.
\end{proof}

%Combining the latter with Propostion~\ref{prop:leftperpforartin} we immediately get the following.

As in Corollary~\ref{rem:explicitL}, we consider the perpendiculars to these objects with respect to all extensions, that is we consider the stable t-structures $(\prescript{\perp}{}{\X}, (\prescript{\perp}{}{\X})^{\perp})$. We obtain the following descriptions of the right adjoints.

\begin{corollary}
\label{prop:explicitRforArtin}
Let $\Lambda$ be an Artin algebra. Then%, %in the notation of Theorem~\ref{theorem.t_from_left_perp},
\begin{enumerate}[label=\emph{(\roman*)}]
\item the right adjoint of the inclusion $\K_{\coac}(\Proj \Lambda) \hookrightarrow \K(\Proj \Lambda)$ is given by
\[
T \mapsto T_{\prescript{\perp}{}{\left(\Lambda\right)}};
\]
\item the right adjoint of the inclusion $\K_{\ac}(\Inj \Lambda) \hookrightarrow \K(\Inj \Lambda)$ is given by
\[
T \mapsto T_{\prescript{\perp}{}{\left(D_\Lambda\right)}};
\]
\item the right adjoint of the inclusion $\Ktac(\Proj \Lambda) \hookrightarrow \K(\Proj \Lambda)$ is given by
\[
T \mapsto T_{\prescript{\perp}{}{\left(\Lambda\oplus(\rho D_\Lambda) \right)}};
\]
\item the right adjoint of the inclusion $\Ktac(\Inj \Lambda) \hookrightarrow \K(\Inj \Lambda)$ is given by
\[
T \mapsto T_{\prescript{\perp}{}{\left(D_\Lambda\oplus(D_\Lambda\otimes_\Lambda^{\mathbb L}D_\Lambda)\right)}}.
\]
\end{enumerate}
\end{corollary}

Observe that if $\Lambda$ is an Artin algebra, then the right $\GProj \Lambda$-approximations in $\Modules\Lambda$ of Corollary~\ref{cor:GProjcontfinite} are also explicitly described by Proposition~\ref{prop:explicitRforArtin}. Indeed, a right $\GProj \Lambda$-approximation of $M$ appears as $\Cok(X^{-1}\to X^0) \to M$ for a right $\Ktac(\Proj \Lambda)$-approximation $X\to \rho M$ of $\rho M$.

\section*{Appendix}

We provide the two proofs that are missing from Section~\ref{section:preliminaries}.

\begin{lemmaappendix1}
   Let $\T$ and $\T'$ be triangulated categories with coproducts and suppose $\T$ is compactly generated. Let $F \colon\T\to\T'$ be exact and coproduct preserving. If $F$ restricts to a fully faithful functor $\T^{\cc}\to \T'^{\cc}$, then $F$ is fully faithful.
\end{lemmaappendix1}

\begin{proof}
   Suppose $F$ restricts to a fully faithful functor $\T^{\cc}\to\T'^{\cc}$. Fix $X \in \T^{\cc}$ and let
   \[
   X_{\ast} = \{B \in \T \vert \,F \colon \T (X,\Sigma^i B) \to \T' (F X, F \Sigma^i B)\text{ is bijective } \forall \text{ } i \in \mathbb Z\}.
   \]
   First, observe that $X_{\ast}$ is a triangulated subcategory of $\T$. Indeed, given a triangle
   \[
   Y \to Y' \to Y'' \to \Sigma Y
   \]
   in $\T$ with $Y, Y'' \in X_{\ast}$, we get the triangle
   \[
   F Y \to F Y' \to F Y'' \to \Sigma F Y
   \]
   in $\T'$ and hence the following diagram with exact rows, revealing $Y' \in X_{\ast}$.
   \begin{center}
   \begin{tikzpicture}[baseline=(current bounding box.center)]
      \matrix(a)[matrix of math nodes,
      row sep=2em, column sep=2em,
      text height=1.5ex, text depth=0.25ex]
      {\cdots & \T(X,Y) & \T(X,Y') & \T(X,Y'') & \cdots  \\
      \cdots & \T'(F X,F Y) & \T'(F X,F Y') & \T'(F X,F Y'') & \cdots   \\};
      \path[->,font=\scriptsize](a-1-1) edge (a-1-2)
      (a-1-2) edge (a-1-3)
      (a-1-3) edge (a-1-4)
      (a-1-4) edge (a-1-5)

      (a-2-1) edge (a-2-2)
      (a-2-2) edge (a-2-3)
      (a-2-3) edge (a-2-4)
      (a-2-4) edge (a-2-5)

      (a-1-2) edge node[right]{$\cong$}(a-2-2)
      (a-1-3) edge (a-2-3)
      (a-1-4) edge node[right]{$\cong$}(a-2-4);
   \end{tikzpicture}
   \end{center}
   Second, $X_{\ast}$ is closed under coproducts. Namely, for each family $\left(B_i \vert\,i\in I\right)$ of objects in $X_{\ast}$ we have
   \[
   \T(X, \coprod_{i\in I} B_i) \cong  \coprod_{i\in I}\T(X,B_i)
   \]
   since $X$ is compact, while
   \[
   \T'(F X, F \coprod_{i\in I} B_i) \cong \T'(F X, \coprod_{i\in I} F B_i) \cong \coprod_{i\in I} \T'(F X, F B_i)
   \]
   as $F$ preserves coproducts and $F X$ is compact, and hence $\coprod_{i\in I} B_i \in X_{\ast}$. Since $\T^{\cc}\subset X_{\ast}$ by assumption, this implies $X_{\ast}=\T$.

   Fix now an arbitrary $B \in \T$ and consider the subcategory
   \[
   B^{\ast} = \{A \in \T \vert \, F \colon \T (\Sigma^i A,B) \to \T' (F \Sigma^i A, F B)\text{ is bijective } \forall \text{ } i \in \mathbb Z \}.
   \]
   We claim that $B^{\ast}=\T$, which would clearly suffice. As above, it suffices to show that $B^{\ast}$ contains $\T^{\cc}$ and is closed under extensions and coproducts. First of all, the above obtained $X_{\ast} = \T$ for each compact $X$ means that $\T^{\cc} \subset B^{\ast}$. Further, $B^{\ast}$ is closed under extensions by an argument similar to the above. Finally, $B^{\ast}$ is closed under coproducts, since for each family $\left( A_i \vert i\in I\right)$ of objects in $B^{\ast}$, we have
   \[
   \T(\coprod_{i\in I} A_i, B) \cong \prod_{i\in I} \T( A_i, B)
   \]
   and on the other hand
   \[
   \T'(F \coprod_{i\in I} A_i, F B) \cong \T'(\coprod_{i\in I} F A_i, F B) \cong \prod_{i\in I} \T'( F A_i, F B). \qedhere
   \]
\end{proof}

\begin{lemmaappendix2}
If $\Lambda$ is an Artin algebra, then the natural monomorphism $\Lambda^{(I)} \to \Lambda^I$ is split for any index set $I$.
\end{lemmaappendix2}
\begin{proof}
We first recall a few notions. Let $R$ be a ring. A monomorphism $X \to Y$ in $\Modules R$ is called \emph{pure} if $X \otimes_R N \to Y \otimes_R N$ remains a monomorphism for each finitely presented left $\Lambda$-module $N$. An $R$-module $M$ is called \emph{pure-injective} if it is injective with respect to pure monomorphisms, i.e.\ given a pure monomophism $f\colon X\to Y$ in $\Modules R$, each morphism $g\colon X\to M$ factors through $f$:
\begin{center}
\begin{tikzpicture}[baseline=(current bounding box.center)]
   \matrix(a)[matrix of math nodes,
   row sep=2em, column sep=2em,
   text height=1.5ex, text depth=0.25ex]
   {X & Y \\
   M & \\};
   \path[->,font=\scriptsize](a-1-1) edge node[above]{$f$}(a-1-2)
   (a-1-1) edge node[left]{$g$}(a-2-1);
   \path[dotted,->,font=\scriptsize](a-1-2) edge (a-2-1);
\end{tikzpicture}
\end{center}

Observe now that the natural monomorphism $\mu_\Lambda\colon \Lambda^{(I)} \to \Lambda^I$ is in fact pure for any indexing set $I$. Indeed, let $N$ be a finitely presented left $R$-module. Then $\Lambda^I \otimes_{\Lambda} N  = N^I$, and we always have $\Lambda^{(I)} \otimes_{\Lambda} N = N^{(I)}$. Hence the induced $\mu_\Lambda\otimes N \colon \Lambda^{(I)}\otimes_\Lambda N \to \Lambda^I \otimes_\Lambda N$ is nothing but the monomorphism $\mu_N \colon N^{(I)} \to N^I$. Since $\Lambda$ is an Artin algebra it has finite length when viewed as a module over its endomorphism ring, which means that $\Lambda^{(I)}$ is pure-injective --- see e.g.\ \cite{MR2792229}.  In particular this means that we have
 \begin{center}
 \begin{tikzpicture}[baseline=(current bounding box.center)]
    \matrix(a)[matrix of math nodes,
    row sep=2em, column sep=2em,
    text height=1.5ex, text depth=0.25ex]
    {\Lambda^{(I)} & \Lambda^I \\
    \Lambda^{(I)} & \\};
    \path[->,font=\scriptsize](a-1-1) edge node[above]{$\mu_\Lambda$}(a-1-2)
    (a-1-1) edge node[left]{$\id$}(a-2-1);
    \path[dotted,->,font=\scriptsize](a-1-2) edge (a-2-1);
 \end{tikzpicture}
 \end{center}
i.e.\ a splitting of $\mu_\Lambda$.
\end{proof}

\bibliographystyle{amsplain}
\bibliography{paperiii}

\providecommand{\bysame}{\leavevmode\hbox to3em{\hrulefill}\thinspace}
\providecommand{\MR}{\relax\ifhmode\unskip\space\fi MR }
\providecommand{\MRhref}[2]{%
  \href{http://www.ams.org/mathscinet-getitem?mr=#1}{#2}
}
\providecommand{\href}[2]{#2}
\begin{thebibliography}{10}

\bibitem{MR2927802}
Takuma Aihara and Osamu Iyama, \emph{Silting mutation in triangulated
  categories}, J. Lond. Math. Soc. (2) \textbf{85} (2012), no.~3, 633--668.
  \MR{2927802}

\bibitem{MR1780017}
Apostolos Beligiannis, \emph{The homological theory of contravariantly finite
  subcategories: {A}uslander-{B}uchweitz contexts, {G}orenstein categories and
  (co-)stabilization}, Comm. Algebra \textbf{28} (2000), no.~10, 4547--4596.
  \MR{1780017}

\bibitem{MR2138374}
\bysame, \emph{Cohen-{M}acaulay modules, (co)torsion pairs and virtually
  {G}orenstein algebras}, J. Algebra \textbf{288} (2005), no.~1, 137--211.
  \MR{2138374}

\bibitem{MR2737805}
\bysame, \emph{On algebras of finite {C}ohen-{M}acaulay type}, Adv. Math.
  \textbf{226} (2011), no.~2, 1973--2019. \MR{2737805}

\bibitem{MR2524651}
Apostolos Beligiannis and Henning Krause, \emph{Thick subcategories and
  virtually {G}orenstein algebras}, Illinois J. Math. \textbf{52} (2008),
  no.~2, 551--562. \MR{2524651}

\bibitem{MR2327478}
Apostolos Beligiannis and Idun Reiten, \emph{Homological and homotopical
  aspects of torsion theories}, Mem. Amer. Math. Soc. \textbf{188} (2007),
  no.~883, viii+207. \MR{2327478}

\bibitem{MR3356832}
Petter~Andreas Bergh, Steffen Oppermann, and David~A. Jorgensen, \emph{The
  {G}orenstein defect category}, Q. J. Math. \textbf{66} (2015), no.~2,
  459--471. \MR{3356832}

\bibitem{MR1214458}
Marcel B\"okstedt and Amnon Neeman, \emph{Homotopy limits in triangulated
  categories}, Compositio Math. \textbf{86} (1993), no.~2, 209--234.
  \MR{1214458}

\bibitem{MR543337}
A.~K. Bousfield, \emph{The {B}oolean algebra of spectra}, Comment. Math. Helv.
  \textbf{54} (1979), no.~3, 368--377. \MR{543337}

\bibitem{MR728454}
\bysame, \emph{Correction to: ``{T}he {B}oolean algebra of spectra''\
  [{C}omment. {M}ath. {H}elv. {\bf 54} (1979), no. 3, 368--377; {MR}0543337
  (81a:55015)]}, Comment. Math. Helv. \textbf{58} (1983), no.~4, 599--600.
  \MR{728454}

\bibitem{buchweitz1986maximal}
Ragnar-Olaf Buchweitz, \emph{Maximal {C}ohen--{M}acaulay modules and
  {T}ate-cohomology over {G}orenstein rings}, unpublished (1986), available at
  http://hdl.handle.net/1807/16682.

\bibitem{MR2501179}
Xiao-Wu Chen, \emph{Singularity categories, {S}chur functors and triangular
  matrix rings}, Algebr. Represent. Theory \textbf{12} (2009), no.~2-5,
  181--191. \MR{2501179}

\bibitem{MR2790881}
\bysame, \emph{Relative singularity categories and {G}orenstein-projective
  modules}, Math. Nachr. \textbf{284} (2011), no.~2-3, 199--212. \MR{2790881}

\bibitem{MR2880676}
\bysame, \emph{The singularity category of an algebra with radical square
  zero}, Doc. Math. \textbf{16} (2011), 921--936. \MR{2880676}

\bibitem{MR3209319}
\bysame, \emph{Singular equivalences induced by homological epimorphisms},
  Proc. Amer. Math. Soc. \textbf{142} (2014), no.~8, 2633--2640. \MR{3209319}

\bibitem{MR3490659}
\bysame, \emph{Singular equivalences of trivial extensions}, Comm. Algebra
  \textbf{44} (2016), no.~5, 1961--1970. \MR{3490659}

\bibitem{MR2236602}
Lars~Winther Christensen, Anders Frankild, and Henrik Holm, \emph{On
  {G}orenstein projective, injective and flat dimensions---a functorial
  description with applications}, J. Algebra \textbf{302} (2006), no.~1,
  231--279. \MR{2236602}

\bibitem{MR1355071}
Edgar~E. Enochs, Overtoun M.~G. Jenda, and Jin~Zhong Xu, \emph{Foxby duality
  and {G}orenstein injective and projective modules}, Trans. Amer. Math. Soc.
  \textbf{348} (1996), no.~8, 3223--3234. \MR{1355071}

\bibitem{MR1432346}
Edgar~E. Enochs, Overtoun M.~G. Jenda, and Jinzhong Xu, \emph{Lifting group
  representations to maximal {C}ohen-{M}acaulay representations}, J. Algebra
  \textbf{188} (1997), no.~1, 58--68. \MR{1432346}

\bibitem{MR1140607}
Werner Geigle and Helmut Lenzing, \emph{Perpendicular categories with
  applications to representations and sheaves}, J. Algebra \textbf{144} (1991),
  no.~2, 273--343. \MR{1140607}

\bibitem{MR1112170}
Dieter Happel, \emph{On {G}orenstein algebras}, Representation theory of finite
  groups and finite-dimensional algebras ({B}ielefeld, 1991), Progr. Math.,
  vol.~95, Birkh\"auser, Basel, 1991, pp.~389--404. \MR{1112170}

\bibitem{MR0222093}
Robin Hartshorne, \emph{Residues and duality}, Lecture notes of a seminar on
  the work of A. Grothendieck, given at Harvard 1963/64. With an appendix by P.
  Deligne. Lecture Notes in Mathematics, No. 20, Springer-Verlag, Berlin-New
  York, 1966. \MR{0222093}

\bibitem{MR2262932}
Srikanth Iyengar and Henning Krause, \emph{Acyclicity versus total acyclicity
  for complexes over {N}oetherian rings}, Doc. Math. \textbf{11} (2006),
  207--240. \MR{2262932}

\bibitem{MR1859047}
Peter Jørgensen, \emph{Spectra of modules}, J. Algebra \textbf{244} (2001),
  no.~2, 744--784. \MR{1859047}

\bibitem{MR2132765}
\bysame, \emph{The homotopy category of complexes of projective modules}, Adv.
  Math. \textbf{193} (2005), no.~1, 223--232. \MR{2132765}

\bibitem{MR2283103}
\bysame, \emph{Existence of {G}orenstein projective resolutions and {T}ate
  cohomology}, J. Eur. Math. Soc. (JEMS) \textbf{9} (2007), no.~1, 59--76.
  \MR{2283103}

\bibitem{MR3031826}
Bernhard Keller and Pedro Nicol\'as, \emph{Weight structures and simple dg
  modules for positive dg algebras}, Int. Math. Res. Not. IMRN (2013), no.~5,
  1028--1078. \MR{3031826}

\bibitem{MR907948}
Bernhard Keller and Dieter Vossieck, \emph{Sous les cat\'egories d\'eriv\'ees},
  C. R. Acad. Sci. Paris S\'er. I Math. \textbf{305} (1987), no.~6, 225--228.
  \MR{907948}

\bibitem{MR1403918}
Maxim Kontsevich, \emph{Homological algebra of mirror symmetry}, Proceedings of
  the {I}nternational {C}ongress of {M}athematicians, {V}ol.\ 1, 2 ({Z}\"urich,
  1994), Birkh\"auser, Basel, 1995, pp.~120--139. \MR{1403918}

\bibitem{MR2157133}
Henning Krause, \emph{The stable derived category of a {N}oetherian scheme},
  Compos. Math. \textbf{141} (2005), no.~5, 1128--1162. \MR{2157133}

\bibitem{MR1191736}
Amnon Neeman, \emph{The connection between the {$K$}-theory localization
  theorem of {T}homason, {T}robaugh and {Y}ao and the smashing subcategories of
  {B}ousfield and {R}avenel}, Ann. Sci. \'Ecole Norm. Sup. (4) \textbf{25}
  (1992), no.~5, 547--566. \MR{1191736}

\bibitem{MR1308405}
\bysame, \emph{The {G}rothendieck duality theorem via {B}ousfield's techniques
  and {B}rown representability}, J. Amer. Math. Soc. \textbf{9} (1996), no.~1,
  205--236. \MR{1308405}

\bibitem{MR1812507}
\bysame, \emph{Triangulated categories}, Annals of Mathematics Studies, vol.
  148, Princeton University Press, Princeton, NJ, 2001. \MR{1812507}

\bibitem{MR2439608}
\bysame, \emph{The homotopy category of flat modules, and {G}rothendieck
  duality}, Invent. Math. \textbf{174} (2008), no.~2, 255--308. \MR{2439608}

\bibitem{MR2680406}
\bysame, \emph{Some adjoints in homotopy categories}, Ann. of Math. (2)
  \textbf{171} (2010), no.~3, 2143--2155. \MR{2680406}

\bibitem{MR3212862}
\bysame, \emph{The homotopy category of injectives}, Algebra Number Theory
  \textbf{8} (2014), no.~2, 429--456. \MR{3212862}

\bibitem{MR2101296}
Dmitri~O. Orlov, \emph{Triangulated categories of singularities and {D}-branes
  in {L}andau-{G}inzburg models}, Tr. Mat. Inst. Steklova \textbf{246} (2004),
  no.~Algebr. Geom. Metody, Svyazi i Prilozh., 240--262. \MR{2101296}

\bibitem{MR2792229}
Mike Prest, \emph{Pure-injective modules}, Arab. J. Sci. Eng. ASJE. Math.
  \textbf{34} (2009), no.~1D, 175--191. \MR{2792229}

\bibitem{MR3274657}
Chrysostomos Psaroudakis, Øystein Skartsæterhagen, and Øyvind Solberg,
  \emph{Gorenstein categories, singular equivalences and finite generation of
  cohomology rings in recollements}, Trans. Amer. Math. Soc. Ser. B \textbf{1}
  (2014), 45--95. \MR{3274657}

\bibitem{MR1027750}
Jeremy Rickard, \emph{Derived categories and stable equivalence}, J. Pure Appl.
  Algebra \textbf{61} (1989), no.~3, 303--317. \MR{1027750}

\bibitem{MR3234500}
Greg Stevenson, \emph{Derived categories of absolutely flat rings}, Homology
  Homotopy Appl. \textbf{16} (2014), no.~2, 45--64. \MR{3234500}

\bibitem{MR1674648}
Amnon Yekutieli and James~J. Zhang, \emph{Rings with {A}uslander dualizing
  complexes}, J. Algebra \textbf{213} (1999), no.~1, 1--51. \MR{1674648}

\end{thebibliography}
\end{document}